\documentclass[]{amsart}
\usepackage{color,latexsym,amssymb,amsmath,amsthm}
\usepackage[dvips]{graphics}

\definecolor{orange}{rgb}{1,0.5,0}

\newtheorem{thm}{Theorem}
\newtheorem{coro}[thm]{Corollary}
\newtheorem{lemma}[thm]{Lemma}
\newtheorem{propo}[thm]{Proposition}
\theoremstyle{definition}
\newtheorem{defn}{Definition}

\title{A discrete Gauss-Bonnet type theorem}
\author{Oliver Knill}
\date{February 11, 2010}
\address{
        Department of Mathematics \\
        Harvard University \\
        Cambridge, MA, 02138 
        }

\subjclass{Primary:   05C10 , 57M15 }
\keywords{Graph theory, Gauss-Bonnet, Curvature}

\begin{document}
\maketitle

\begin{abstract}
We discuss a curvature theorem for subgraphs of the flat triangular tessellations of the plane. These 
graphs play the analogue of "domains" in two dimensional Euclidean space. We show that the 
Pusieux curvature $K(p) = 2 |S_1(p)| - |S_2(p)|$ satisfies
$\sum_{p \in \delta G}  K(p) = 12 \chi(G)$, where $\chi(G)$ is the Euler characteristic of the graph, 
$\delta G$ is the boundary of $G$ and where $|S_r|$ the arc length of the sphere of radius $r$ in $G$. 
This formula can be seen as a discrete Gauss-Bonnet formula or Hopf Umlaufsatz.
\end{abstract}

\begin{center}
Dedicated to Ernst Specker to his 90th birthday. 
\end{center}


\section{Introduction}

For a domain $D$ in the plane with smooth boundary $C$, the {\bf Gauss-Bonnet formula} or {\bf Umlaufsatz}
$\int_{C} K(s) \; ds = 2\pi \chi(D)$ relates the curvature $K(s)$ of the boundary curve
with the Euler characteristic $\chi$ of 
the region. For a simply connected region $G$ for which the boundary is a simple closed
curve, the total boundary curvature is $2\pi$. This Gauss-Bonnet type result is a form of {\bf Hopf's Umlaufsatz} 
and relates a differential geometric quantity, the boundary curvature, 
with a topological invariant, the Euler characteristic. 
In differential geometry, curvature needs a differentiable structure, while Euler characteristic does not. 
It is the transcending property between different mathematical branches which makes Gauss-Bonnet 
type results interesting.  \\

We prove here a discrete version of a "Hopf Umlaufsatz" \cite{hopf35} which is of combinatorial nature;
curvature is an integer. The result applies to special two dimensional graphs which are part of a flat two dimensional 
background graph $X$, where the 
dimensionality is defined inductively. While the Euler characteristic is a topological notion,
we need "smoothness assumptions" to equate the total boundary curvature with the Euler characteristic. \\

The curvature, we consider here is $K(p) = 2 |S_1| - |S_2|$, where $S_r(p)$ is the arc length of the sphere $S_r(p)$
at the point $p$.  The sphere $S_r$ is a subgraph of $G$ with vertices of all points of distance $r$ and edges 
consisting of pairs $(q,q')$ in $S_r(p)$ such that $q$ and $q$ have distance $1$. 
As we will explore elsewhere, for many compact two-dimensional graphs $G$ without boundary, like triangularizations
of polyhedra with 5 or 6 faces, the integral of the curvature
$K(p) = 2 |S_1|-|S_2|$ over the entire graph is $60 \chi(D)$. The "smoothness" assumptions are more subtle than 
corresponding results for $K_1(p) = 6 - |S_1|$.  For the later, the result $\sum_{p \in G} K_1(p) = \chi(G)$ is 
essentially a reformulation of Euler's formula and holds for any "two dimensional graph" with or without boundary.
We will look at a relation between the "first order curvature" $K_1$ and second order curvature $K$ at the end of 
this article. \\


The main result in this paper is the formula
$\sum_{p \in \delta G} K(p) = 12 \chi(G)$ which holds for discrete domains $G$ and for a {\bf second order} curvature $K$.
To do so, we need to specify precisely what a "smooth domain" is. \\


The background lattice $X$ plays here the role of the two-dimensional plane. Its vertices can be realized as
the set of points $\{ k (1,0) + l (1,\sqrt{3})/2 \; | \; k,l \; {\rm integers} \; \}$. The edges are formed by the set of 
pairs for which the Euclidean distance is $1$.  In the infinite graph $X$, every point $p$ has 6 neighbors. Together with edges formed
by neighboring vertices, these points form the {\bf unit sphere} $S_1(p)$, a subgraph of $X$.
Similarly, any sphere $S_2(p)$ of radius $2$ in this discrete plane has length $|S_2|=12$.  
The {\bf curvature} $K = 2 |S_1| - |S_2|$ is zero at every point of the background lattice $X$.  \\

\begin{center}
\scalebox{0.90}{\includegraphics{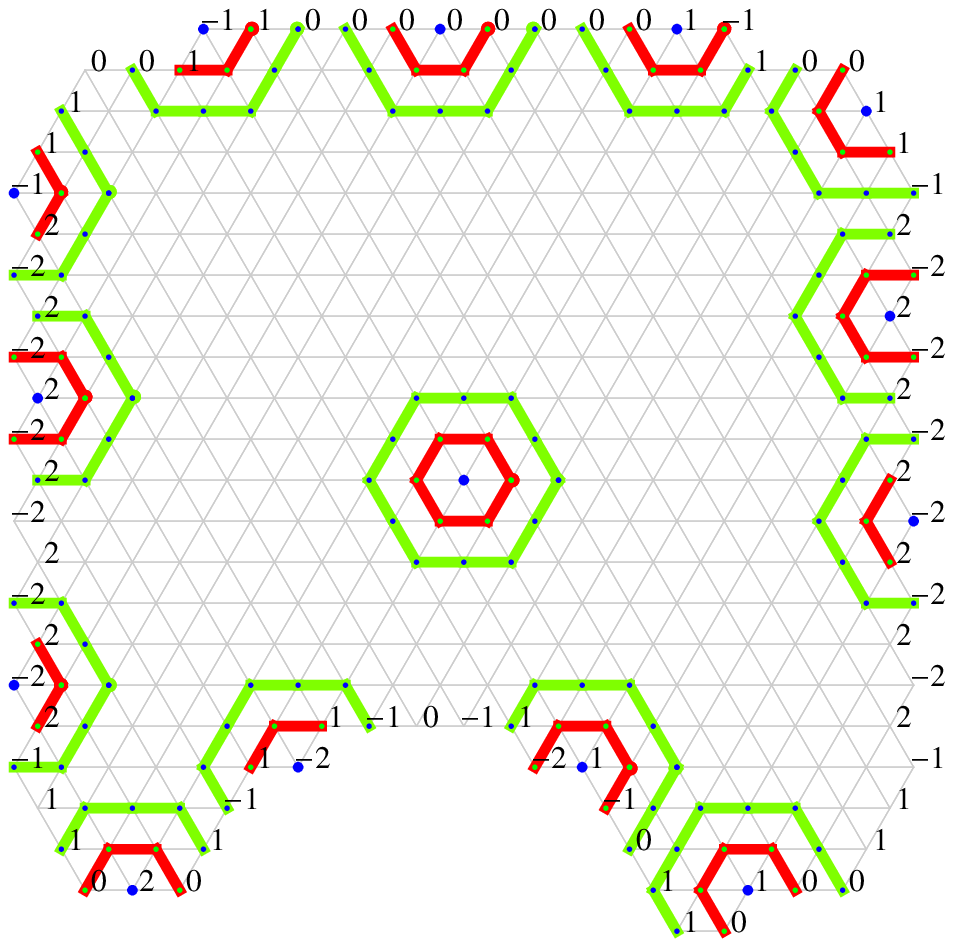}}\\
\end{center}
{\bf Figure:} Curvature computation. The numbers near each vertex indicate the curvature of the point.
At each of a few chosen points, we have drawn the spheres of radius $1$ and $2$ in $G$.
Adding up the curvatures over the boundary gives 12. If a point has as a neighborhood a 
disc of radius $2$, the curvature is zero.  

\section{Topology of the planar triangular lattice}

A finite subset $G$ of the triangular lattice $X$ defines a {\bf graph} $(V,E)$, where $V \subset X$ is the 
set of vertices in $G$ and where $E$ is a subset of edges $(p,q)$ in $X$, pairs in $X$ have distance $1$ within $X$. 
We start by defining a dimension for graphs. To our best knowledge, this notion seems not yet have
appeared, even so in the graph theory literature, several notions of dimension exist. 
The definition of dimension is inductive and rather general and does not require the graph to be a subset of $X$. 

\begin{defn}
A {\bf sphere} $S_r(p)$ in the subgraph $G$ of $X$ whose vertices consists the set of points in $G$ which have geodesic 
distance $r$ to $p$, normalized, so that adjacent points have distance $1$ within $G$. The edges of the 
sphere graph $S_r$ are all pairs $(p,q)$ with $p,q \in S_r(p)$ for which $(p,q)$ is in $G$.  
A {\bf disc} $B_r(p)$ in the graph $G$ is the set of points $q$ which have distance $d(q,p) \leq r$ in $G$.
\end{defn}

\begin{defn}
A {\bf vertex $p$ of a graph $G=(V,E)$} is called {\bf 0-dimensional}, if $p$ is not connected to any other vertex.
A {\bf subset $G$} of $X$ is called {\bf $0$-dimensional} if every point of $G$ is $0$-dimensional in $G$. 
Zero-dimensionality for a graph 
means that it has no edges. A point $p$ of $G$ is called {\bf 1-dimensional} if $S_1(p)$ is 
$0$-dimensional, where $S_1(p)$ is the unit sphere of $p$ within $G$. A finite subset $G$ of $X$
is called {\bf 1-dimensional} if any of the points in $G$ is $1$-dimensional. A point $p$ of $G$ is called $2$ 
dimensional, if $S_1(p)$ is a one-dimensional graph.
A subset $G$ of $X$ is called $2$-dimensional, if every vertex $p$ of $G$ is a $2$-dimensional point. 
\end{defn}

The dimension does not need to be defined. For example, a point which has a sphere which contains of one and zero dimensional 
components has no dimension. One could define inductively a {\bf fractional dimension} 
by adding $1$ to the average fractional dimensions of the points on the unit sphere. \\

As an illustration of the notion of dimension,
lets look at the platonic solids as graphs. The cube and the dodecahedron are one dimensional. The
isocahedron and octahedron are two dimensional. The tetrahedron is three dimensional, because the unit sphere
of each point is 2 dimensional. The cube and dodecahedron become two dimensional after kising (stellating) their faces. 
The tetrahedron becomes 2 dimensional after truncating corners. 

\begin{defn}
A point $p$ in $G$ called an {\bf interior point} of $G$ if the sphere $S_1(p)$ in the graph $G$ is the 
same than the sphere $S_1(p)$ in the background lattice $X$. In other words, for an 
interior point, the sphere $S_1(p)$ is a one-dimensional graph without boundary.
\end{defn}


\begin{defn}
A point $p$ of a two-dimensional graph $G$ is a {\bf boundary point} of $G$, if 
it is not an interior point in $G$ but has a neighbor in $G$ which is an interior point. 
\end{defn}

For an interior point, the sphere $S_1(p)$ is a closed circle, 
for a boundary point, the sphere $S_1(p)$ is a union of one-dimensional arcs. 

\begin{defn}
The {\bf boundary} of $G$ is the set of boundary points of $G$.  
The {\bf interior} of $G$ is the set of interior points of $G$.
\end{defn}

{\bf Remarks:}  \\
a) The set of subsets $\{ A \subset {\rm int}(G) \; \} \cup \{ G \; \} \; $ defines a 
topology on $G$ such that the interior of $G$ is open and the boundary $\delta G$ is closed.   \\
b) The interior of a two dimensional graph $G$ is not necessarily a 2 dimensional graph. 
The disc of radius $1$ in $X$ for example is has a single interior point so that the interior is zero-dimensional. \\
c) Two dimensionality of a graph has no relation with "being planar". 
There are planar graphs like the tetrahedral graph which is three dimensional in our sense but which is planar.
And there are graphs like triangularizations of a torus, which are two dimensional but not planar.  \\
d) Topologically, one can show that the triangular graph $X$ is the only simply-connected two-dimensional
flat graph without boundary. \\


\begin{defn}
We call a subset $G$ of $X$ a {\bf domain} if the following 5 conditions are satisfied: 

\begin{center}
\parbox{12cm}{
\noindent
(i) $G$ is a two-dimensional subgraph of $X$. \\
(ii) Every point of $G$ is either an interior point or a boundary point. \\
(iii) The set of boundary points in $G$ is a one-dimensional graph.  \\
(iv) If two vertices $p,q$ in $G$ have distance $1$ in $X$, then $(p,q)$ is an edge in $G$. \\
(v) Two interior points in $G$ with a common boundary points have distance 1 or are both adjacent 
to a third interior point. 
}
\end{center}
\end{defn}

The conditions (i),(ii), (iii) are natural. Condition (iv) assures that no unnatural "fissures" can exist.
Condition (v) assures that the connectivity topology of the domain and the connectivity topology 
of the interior set are the same. \\

\begin{defn}
A domain is called a {\bf finite domain}, if it is a finite graph which is a domain.
A domain is called {\bf a smooth domain}, if it is a domain and its complement is a domain too.
\end{defn}

{\bf Remarks:} \\
a) We could additionally require the interior of a domain to be two-dimensional but we do not need that. 
Actually, the proof of the main theorem becomes simpler if we do not make this assumption. It would just 
lift a difficulty on a different level. For us it will be important to look at the dimension of points 
in the interior of $G$. \\
b) Some of these conditions for "domains" have analogues in the continuum, where they are necessary 
for the classical Gauss-Bonnet to be true:
we can not have "hairs" sticking out of the domain for example. 
The closure of the complement of a domain is a domain
too and we can not just leave out part of the boundary. Also in the continuum, it should not happen that 
parts of domains are tangent to each other. We also can not allow the boundary to be two-dimensional, 
like for the Mandelbrot set. \\
c) For a smooth domain, we can look at the interior $H' = {\rm int}(G')$ of the complement $G'$ of ${\rm int}(G)$. 
Then, the boundaries satisfy $\delta(G) = \delta(G')$. The three sets ${\rm int}(G')$,
${\rm int}(G)$ and $\delta(G)=\delta(G')$ partition the graph $X$. 

\begin{center}
\scalebox{0.55}{\includegraphics{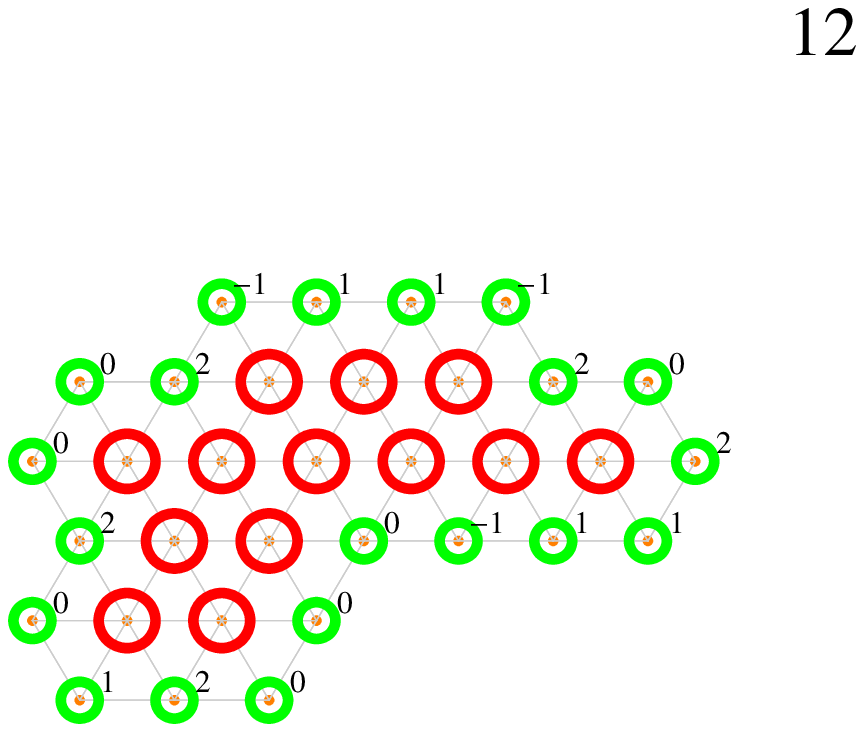}} 
\scalebox{0.55}{\includegraphics{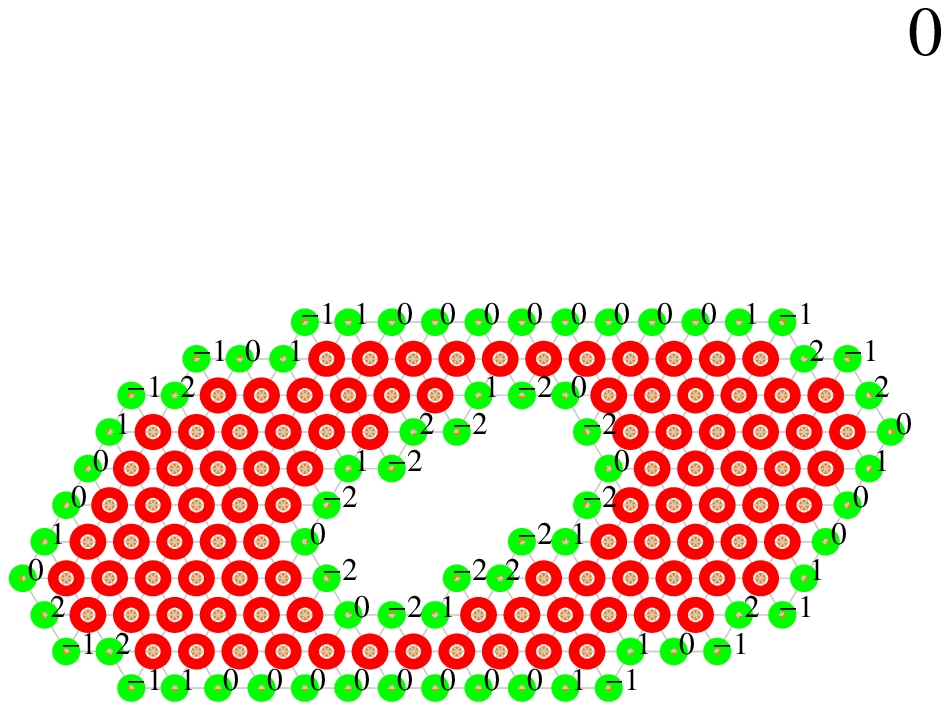}} 

\scalebox{0.55}{\includegraphics{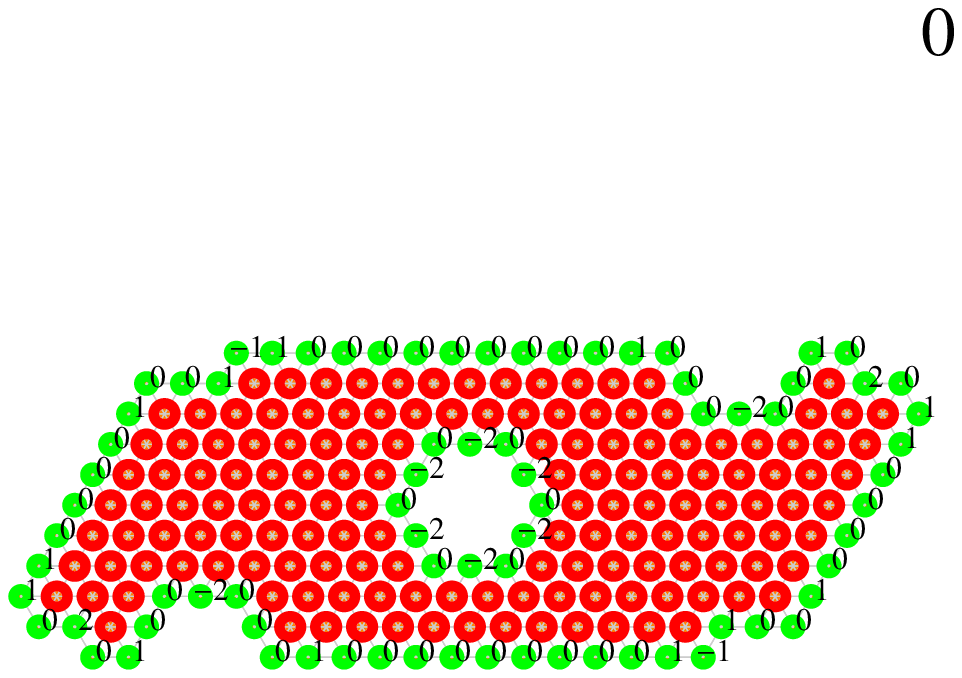}} 
\scalebox{0.55}{\includegraphics{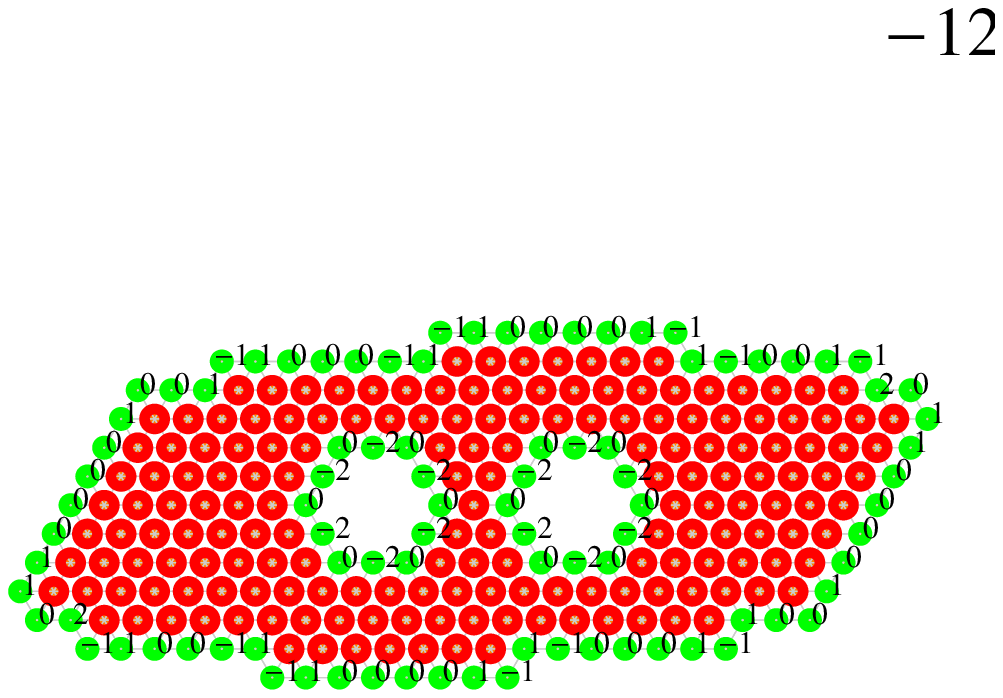}} 
\end{center}
{\bf Figure:} Examples of domains. The number in the upper right corner is the total boundary curvature
of the domain. 

\begin{center}
\scalebox{0.55}{\includegraphics{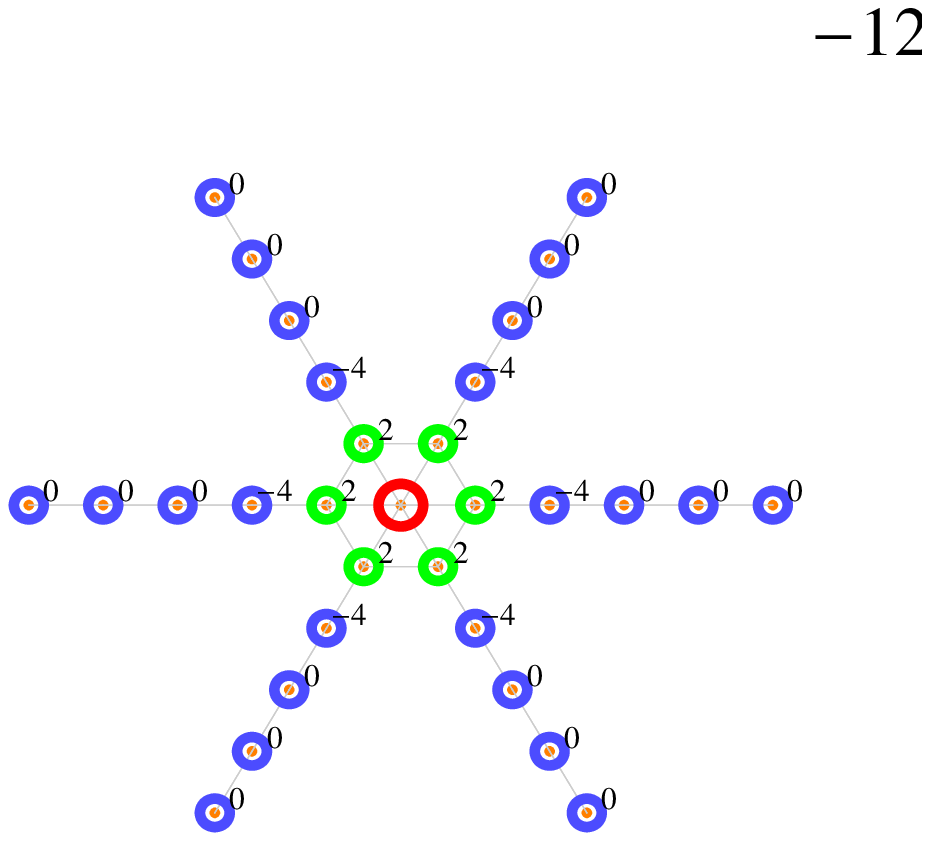}} 
\scalebox{0.55}{\includegraphics{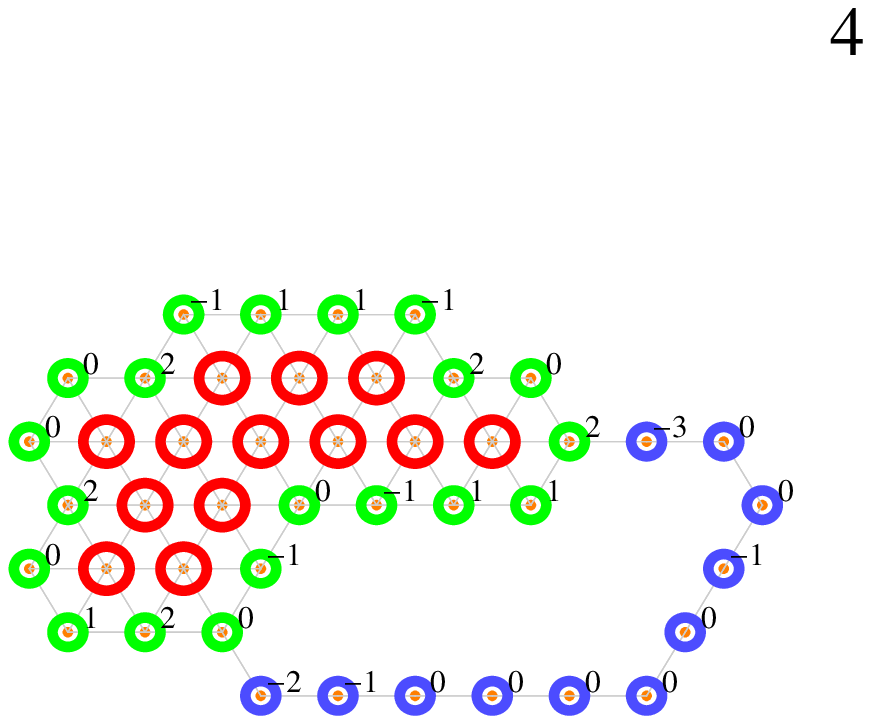}} 
\end{center}
{\bf Figure:} Examples of graphs which are not domains. To the left, a set with 2, 1 and 0 dimensional points.
It violates conditions (i). The second example is a set with both 2 and 1 dimensional points. 

\begin{center}
\scalebox{0.55}{\includegraphics{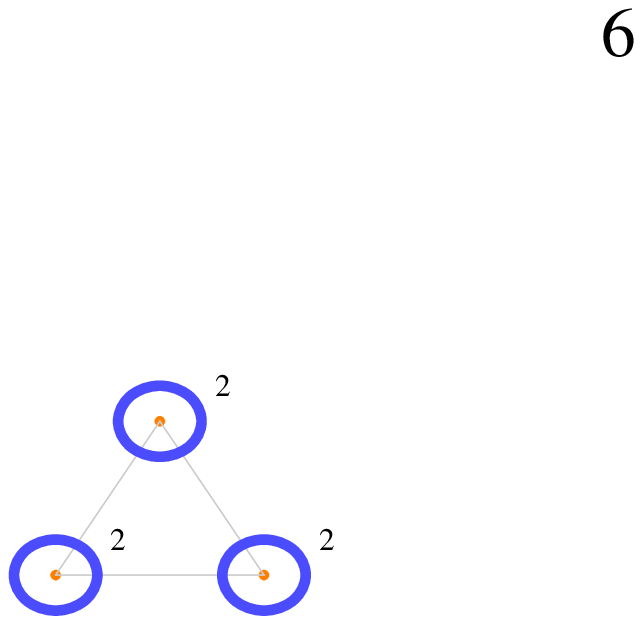}} 
\scalebox{0.55}{\includegraphics{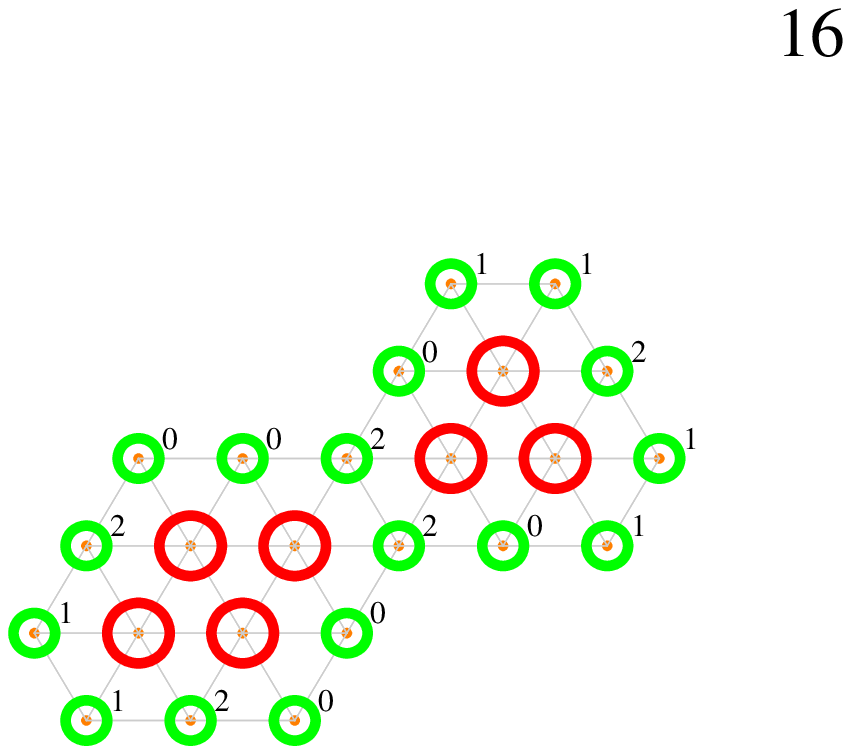}} 
\end{center}
{\bf Figure:} The left example is a two-dimensional set with no interior points and no boundary points. 
It violates condition (ii). The second one violates (v).

\begin{center}
\scalebox{0.55}{\includegraphics{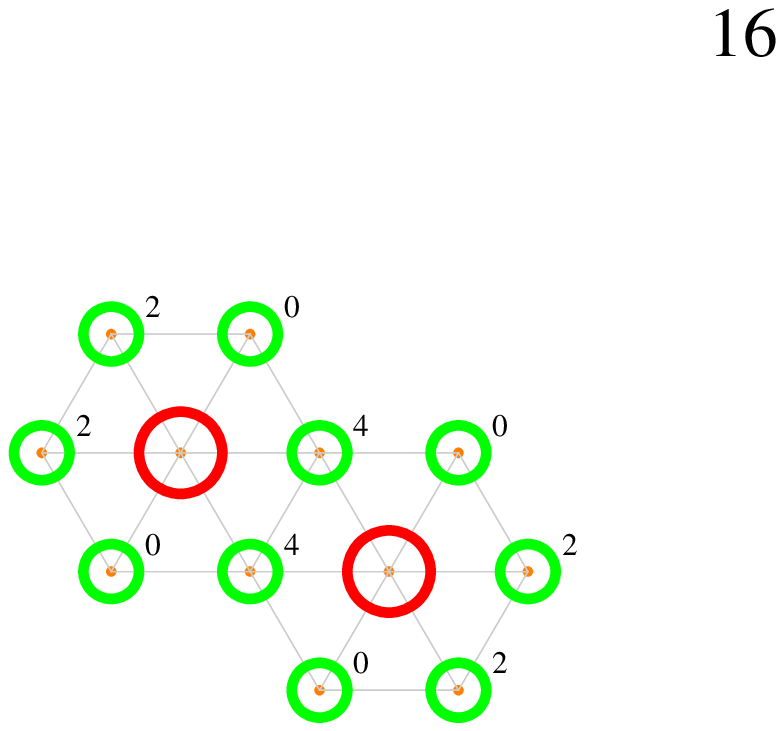}} 
\scalebox{0.55}{\includegraphics{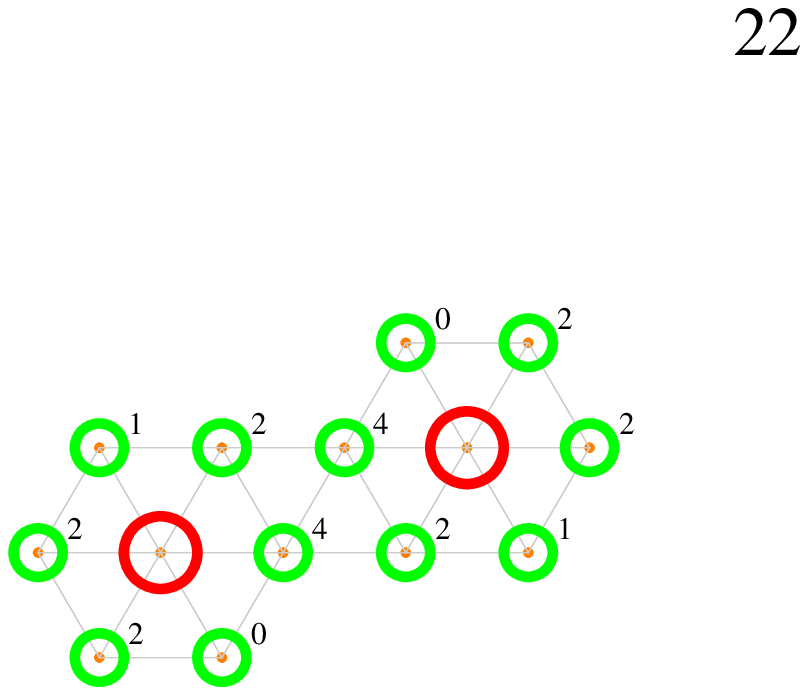}} 
\end{center}
{\bf Figure:} Examples of graphs which are not domains. 
The first violates (v), the third violates (iii).

\begin{center}
\scalebox{0.55}{\includegraphics{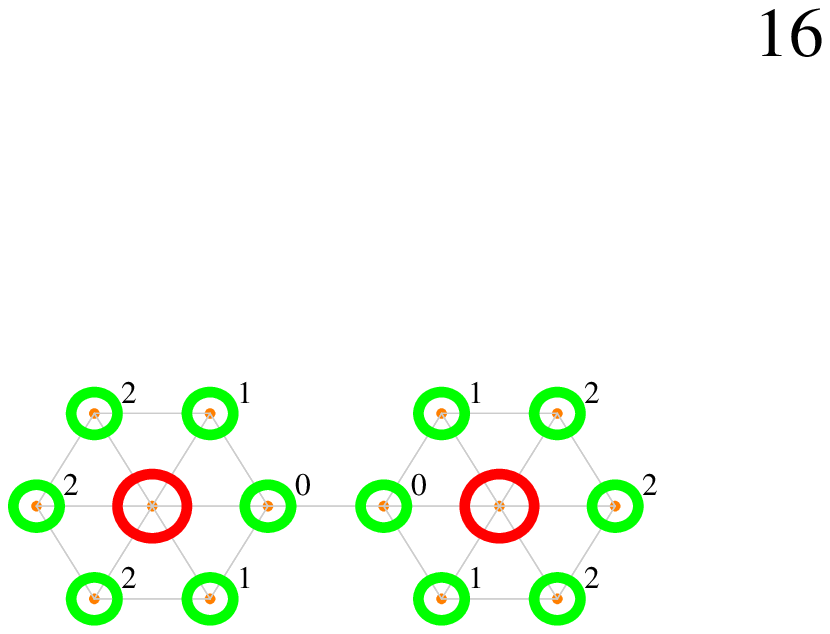}}
\scalebox{0.55}{\includegraphics{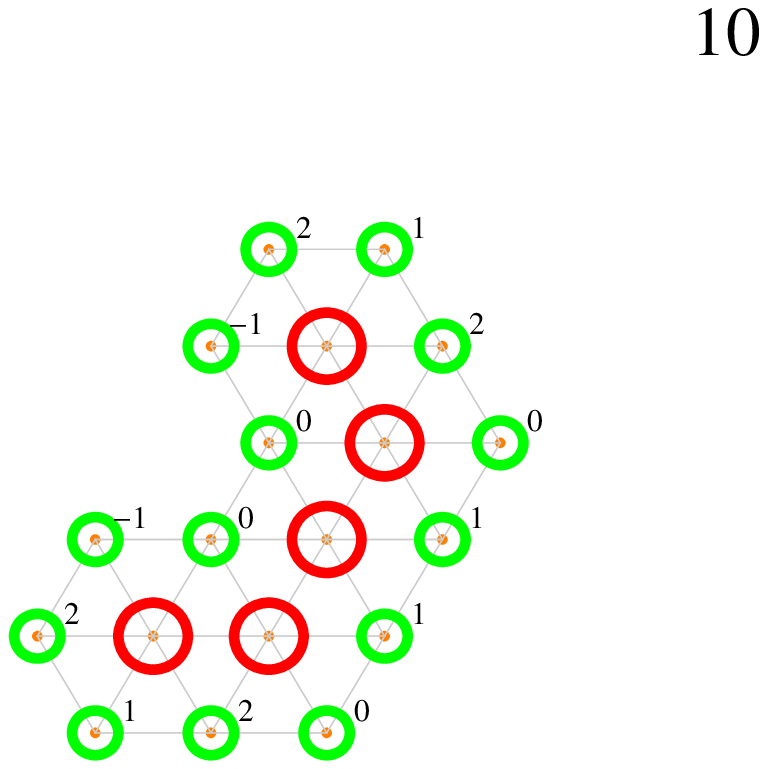}} 
\scalebox{0.55}{\includegraphics{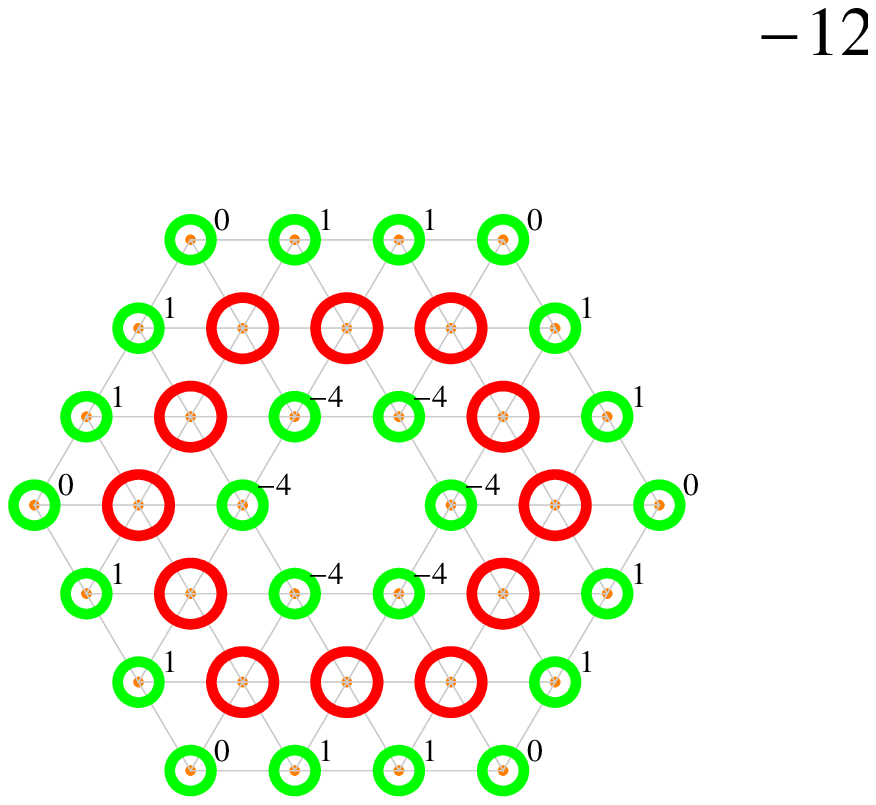}}
\end{center}
{\bf Figure:} Domains which are not smooth domains. The curvature of the first was computed
while assuming the nearest neighbor connection to be an edge as required by condition (iv). 
If the connection is not in place (violating (iv)), the total curvature would be 24. \\


The following lemma allows us to deal more efficiently with eligible regions and eliminates
many subsets which are not regions. It says that the set of interior points determines 
the region as well as its boundary. 

\begin{lemma}
Let $G$ be a domain and $H={\rm int}(G)$ be the set of interior points of $G$.
Then $G= \bigcup_{q \in H} B_1(q)$,
where $B_1(q)$ is the disc of radius $1$ in $X$. Especially, the interior set $H={\rm int}(G)$ 
determines the domain $G$ completely.
\end{lemma}

\begin{proof}
If a point $p$ is in $G$, then it is either an interior point or a point adjacent to an 
interior point. Therefore $G \subset \bigcup_{q \in H} B_1(q)$. On the other hand, if 
$p$ is in $\bigcup_{q \in H} B_1(q)$, then $p \in B_1(q)$ for some $q$. Because $q \in G$
and $S_1(q) \in G$ by definition of being an interior point, we have $p \in G$. 
\end{proof}

{\bf Remark:} For a simply connected region, also the boundary of $G$ determines the region, but we
do not need that.

\section{Curvature} 

\begin{defn}
Let $|S_r(p)|$ denotes the number of edges in the sphere $S_r(p)$.
We call it the {\bf arc length} of the sphere $S_r$. 
\end{defn}

{\bf Remarks.}  \\
a) Note that $|S_1|$ is not necessarily the number of vertices in $S_1$. 
Similarly, $|S_2|$ is the number of edges in $S_2$ which is not always equal to 
the number of vertices in $S_2$.  \\
b) The sphere $|S_1|$ does not necessarily have to be connected, nor does it have
to have a defined dimension. It could be a union of a segment and a point for example.

\begin{defn}
The {\bf curvature} of a boundary vertex $p$ in a region $G$ is defined as
$$  K(p) = 2 |S_1(p)| - |S_2(p)|  \; . $$
The {\bf curvature} of a finite domain $G$ is the sum of the curvatures over the boundary.
\end{defn}

{\bf Remarks.}  \\
a) This definition is motivated by differential geometry since one can derive an 
analogue formulas in the continuum 
$K = \lim_{r \to 0}  \frac{2 |S_{r}| - |S_{2r}|}{2 \pi r^3}$ for a point on the 
boundary curvature of a region. \\ 
b) Note that as defined, $S_2(p)$ refers to the 
geodesic circle of radius $2$ {\bf in} $G$ and not in $X$ so that every point $q \in G$ of distance 
$2$ in $X$ to $p$ belongs to $S_2(p)$ whether there is a connection within $G$ from $p$ to $q$ or not. 
The reason for this choice is that we do want the curvature definition to be nonlocal. 
This subtlety will not matter since for the definition of smooth curve, we anyhow disallow 
situations where points have a large distance within $G$ but small distance in $X$. 

\begin{center}
\scalebox{0.55}{\includegraphics{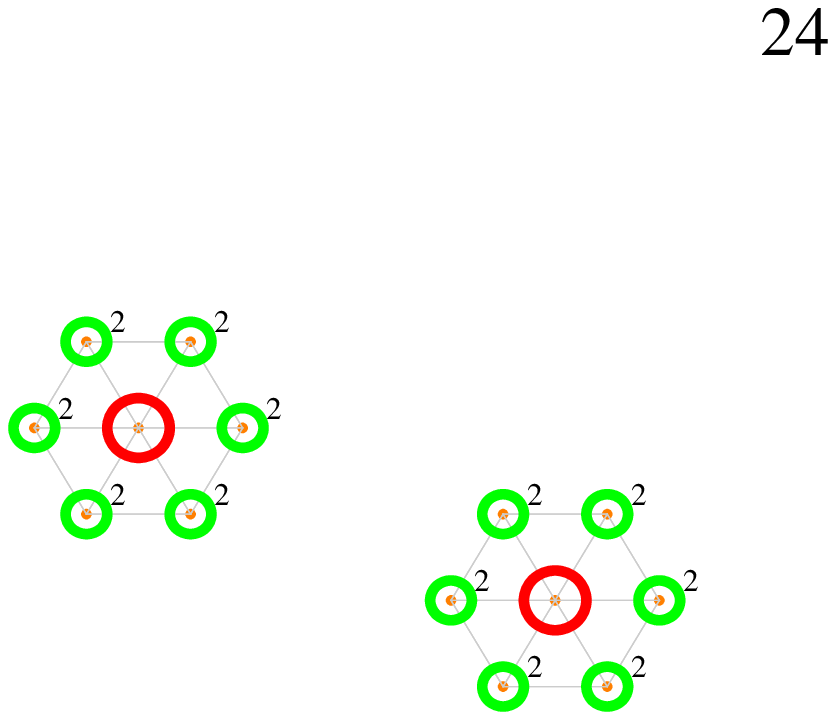}}
\scalebox{0.55}{\includegraphics{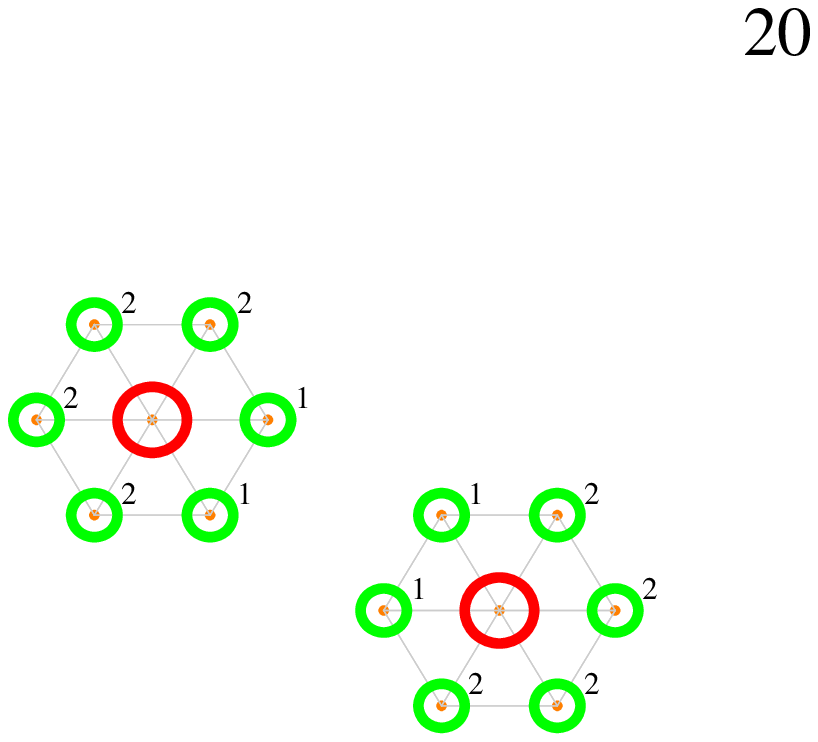}}
\end{center}
{\bf Figure:} The first picture is a smooth domain. It is not simply connected although
the two parts $S_1,S_2$ of the regions are at first separated enough to get a curvature 24. 
In the second case $S_2$ "feels" part of the other region $S_1$ and the curvature is not a multiple of $12$. 

\begin{center}
\scalebox{0.55}{\includegraphics{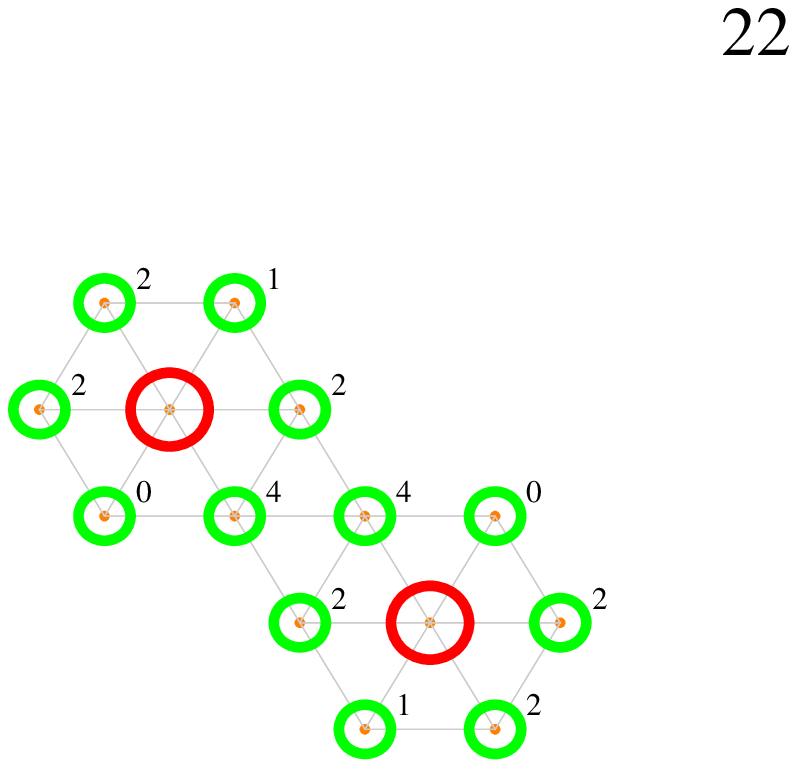}}
\scalebox{0.55}{\includegraphics{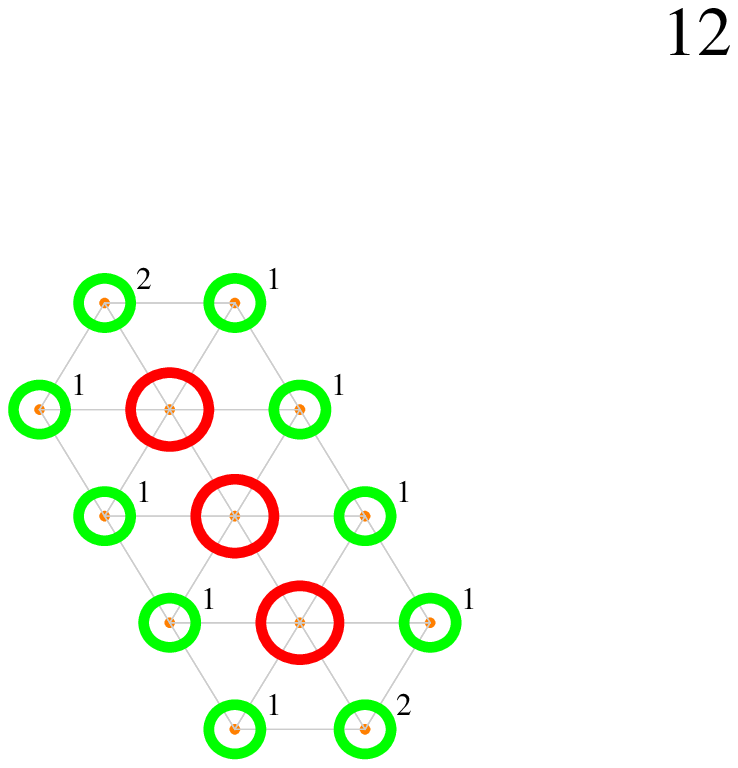}}
\end{center}
{\bf Figure:} In the first picture, the domain is still not smooth 
because the complement is not a domain. The last example is a smooth domain.
It has become simply connected. 

\begin{defn}
A {\bf curve} $\gamma$ in a smooth domain $G$ is a sequence of points $x_0, \dots, \; x_n$ 
in the interior of $G$ such that $d(x_i,x_{i+1})=1$ and consequently $(x_i,x_{i+1})$ is an edge of
$G$.  A curve is a {\bf closed curve} if $x_0=x_n$. In graph theory, a curve is called a chain. 
It is a {\bf nontrivial closed curve} if its length is larger than $1$. It is called a {\bf simple closed curve}, 
if all points $x_0,x_1, \dots ,x_{n-1}$ are different and $x_0=x_n$. 
\end{defn}

\begin{defn}
A domain is called {\bf simply connected} if every closed curve $\{ x_1, \dots, x_n  \; \}$ 
in the interior $H$ of $G$ can be deformed to trivial closed curve within $G$, where a {\bf deformation} of 
a curve within $G$ consists of a composition of finitely many elementary 
deformation steps
$\{ x_1, \dots, x_n \} \to \{ y_1, \dots, y_n \; \}$ with 
$\sum_i d(x_i,y_i)=1$ and such that $x_i,y_i$ are in $H$.
As in the continuum, simply connectedness means that 
any closed curve in the interior of $G$ can be deformed to a point within the interior of $G$. 
\end{defn}


\section{The curvature 12 theorem}

Our main result of this paper is a discrete version of the "Umlaufsatz". It will be generalized
to more general domains below. 

\begin{thm}[Curvature 12 Umlaufsatz]
The total boundary curvature of a finite, smooth and simply connected domain $G$ is $12$. 
\end{thm}

\begin{proof}
For the proof, it suffices to look at local deformations. We start with an arbitrary simply connected 
smooth region $G$ and find a procedure to remove interior points near the boundary while keeping the simply connectedness
property and keeping also the curvature the same. Removing one point only affects the 
curvatures in a disc of radius $2$ so that only finitely many cases need to be studied:

\begin{lemma}[Curvature is local]
Let $G_1,G_2$ be two regions and $p$ be a point in both $G_1$ and $G_2$. 
Let $U_1 = B_2(p) \subset G_1$ and $U_2 =  B_2(p) \subset G_2$ be the discs
of radius $2$ in $G_1$ and $G_2$ respectively. Define $H_i = G_i \setminus \{p\}$. 
If $U_1=U_2$, then 
$$ \sum_{p \in H_1} K(p) - \sum_{p \in H_2} K(p) = \sum_{p \in G_1} K(p) - \sum_{p \in G_2} K(p)  \; . $$
\end{lemma}

In other words, if we remove a point from a region, then the total curvature-change can be read off from 
the curvature-changes in a disc of radius $2$.  \\

We could check all possible configurations in discs of radius $2$ and compare the total curvature 
before and after the center point is removed. We indeed checked with the 
help of a computer that in all cases, where the total boundary curvature changes, the number of local 
connectivity components of the interior has changed or the complement has become non-smooth near the 
removed point. These experiments  helped us also to get the conditions what a domain is.  \\

But checking all possible local deformations is not a proof. We also need to know that there is always a point 
which we can remove without changing the topology of $G$ or its complement.  \\

It turns out that this question is of more global nature. 
Take a ring shaped region for example which has a one dimensional interior. No point can be removed
without the curvature to change. The key is to look at the dimension of points in the interior of $G$ and distinguish points which 
are one-dimensional in ${\rm int}(G)$ and points which are two-dimensional in ${\rm int}(G)$.
A zero dimensional interior means for a simply connected region 
that the graph is the disc of radius $1$ in $X$. By removing interior points, we want to reach this situation. 

\begin{center}
\scalebox{0.85}{\includegraphics{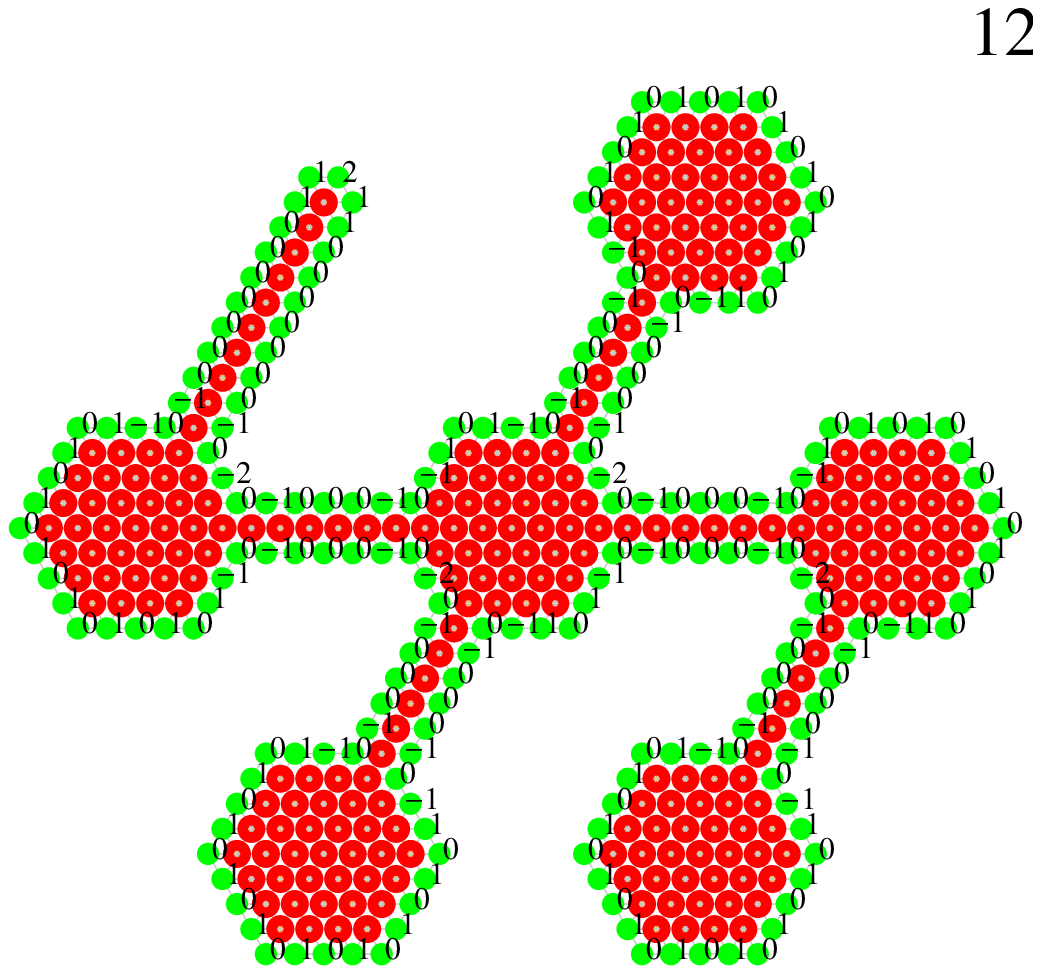}}
\end{center}
{\bf Figure:} Pruning a tree, a simply connected domain. To reduce a region, we have to trim the tree,
removing alternatively two-dimensional interior points and one-dimensional interior points until only 
one interior point is left.  \\

We are allowed to look at the {\bf topology of interior points} because ${\rm int}(G)$ defines $G$ by the above
lemma, it is enough to check what happens if we remove interior points. Our goal is to show: \\ 

\begin{propo}[Trimming a tree] 
For any simply connected smooth region $G$ for which the interior set $H$ has more than one point, 
it is possible to remove an interior point $p$ from $H$, 
such that the new region defined by $H \setminus \{p\}$ remains a simply connected smooth 
region with one interior point less and such that the curvature does not change. 
\end{propo}

The theorem follows from this proposition. Lets introduce some terminology: 

\begin{defn}
Given a smooth, simply connected region $G$ with interior $H$. Denote by $H_1$ the points in $H$ which are one dimensional
in $H$. Similarly, call $H_2$ the set of points in $H$ which are two dimensional in $H$. Connected components of $H_1$
are called either {\bf branches} or {\bf bridges}.  Connected components of $H_2$ are called {\bf ridges}. 
A {\bf branch} of $G$ is a connected component of $H_1$ for which at least one point has only one interior neighbor.
All other connected components of $H_1$ are called bridges. 
\end{defn}

\begin{center}
\scalebox{0.85}{\includegraphics{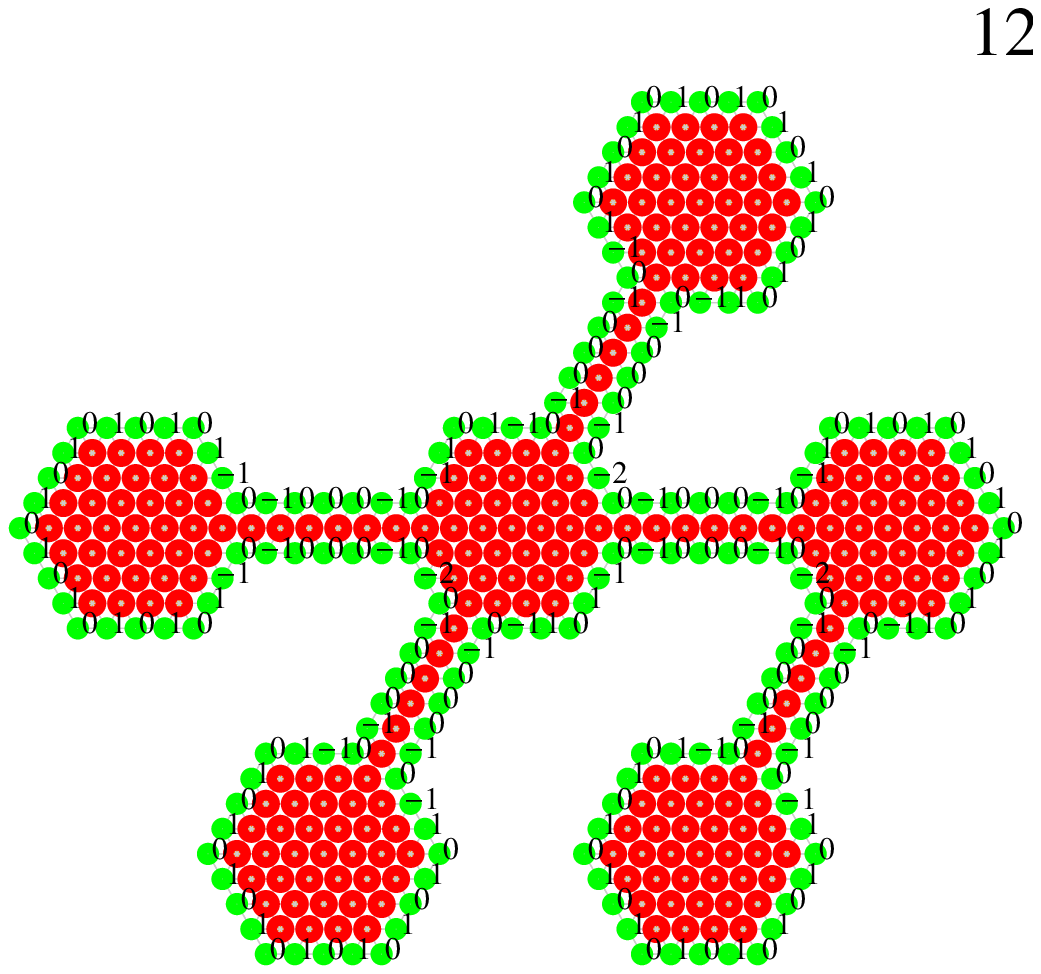}}
\end{center}
{\bf Figure:} A simply connected region with ridges, bridges. All branches have been 
pruned. Now, we have to start etching the ridges. We have the choice of 4 end ridges
here. The simply connectivity assures that there is an end ridge. 

\begin{center}
\scalebox{0.55}{\includegraphics{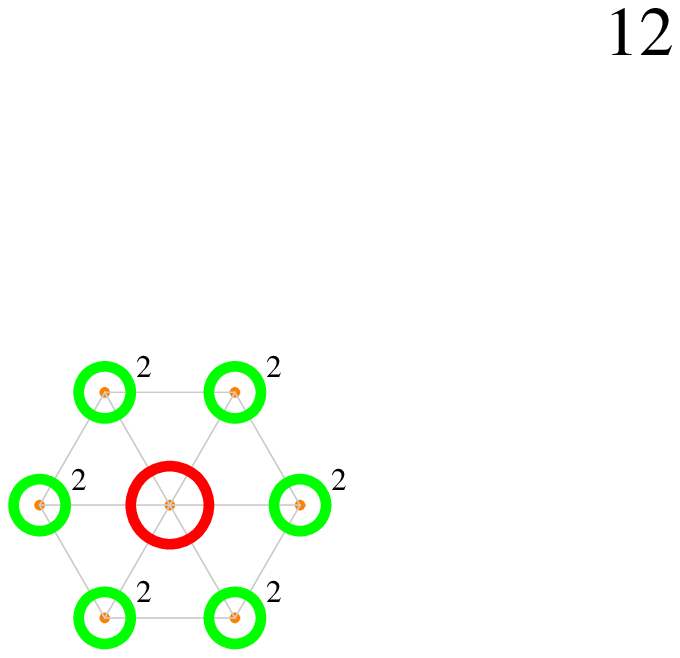}}
\end{center}
{\bf Figure:} The only situation, where we can 
not trim any more one dimensional branches nor two-dimension ridges.  \\

The set ${\rm int}(G)$ is the union of points which are two-dimensional in ${\rm int}(G)$  and points which are 
one-dimensional in ${\rm int}(G)$. We will use two procedures called {\bf pruning} and 
{\bf etching} to make the region smaller. The pruning procedure removes a one-dimensional interior point at branches.
The etching procedure removes a two-dimensional interior point at ridges. \\

\begin{center}
\scalebox{0.85}{\includegraphics{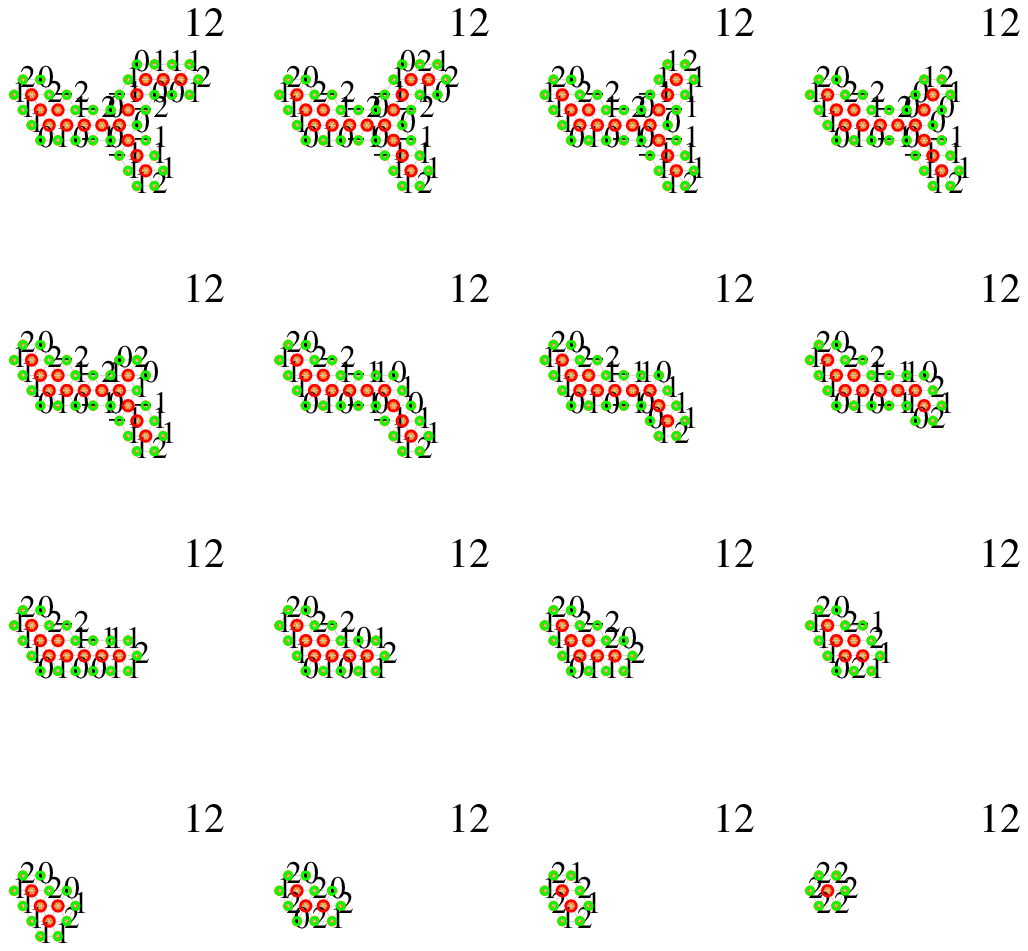}}
\end{center}
{\bf Figure:}  Pruning reduces the lengths of branches. Since curvature is local, we only need to 
check for a few end situations that the total curvature does not change.  \\

Lets start with the {\bf pruning procedure} which removing interior points which are one-dimensional in ${\rm int}(G)$.
It allows us to remove one-dimensional branches until we can no more reduce one-dimensional points in ${\rm int}(G)$.
Removing one-dimensional parts will make sure that there will be a two-dimensional ridges ready for 
the etching procedure. Here are the situations which can occur locally at a point of a branch. 

\begin{center}
\scalebox{0.55}{\includegraphics{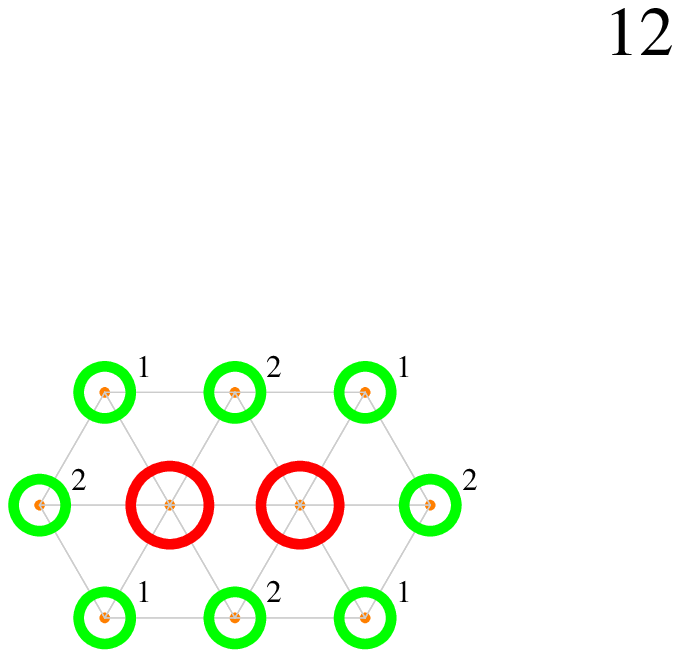}}
\end{center} 
{\bf Figure:} A one-dimensional point which has 1 interior neighbor. After removing a boundary point,
we end up with region 0.

\begin{center}
\scalebox{0.55}{\includegraphics{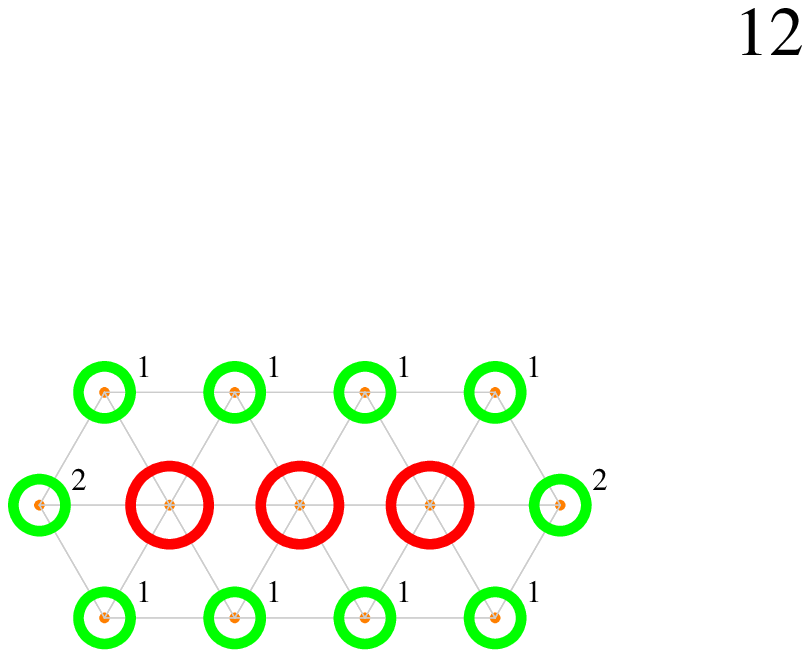}}
\scalebox{0.55}{\includegraphics{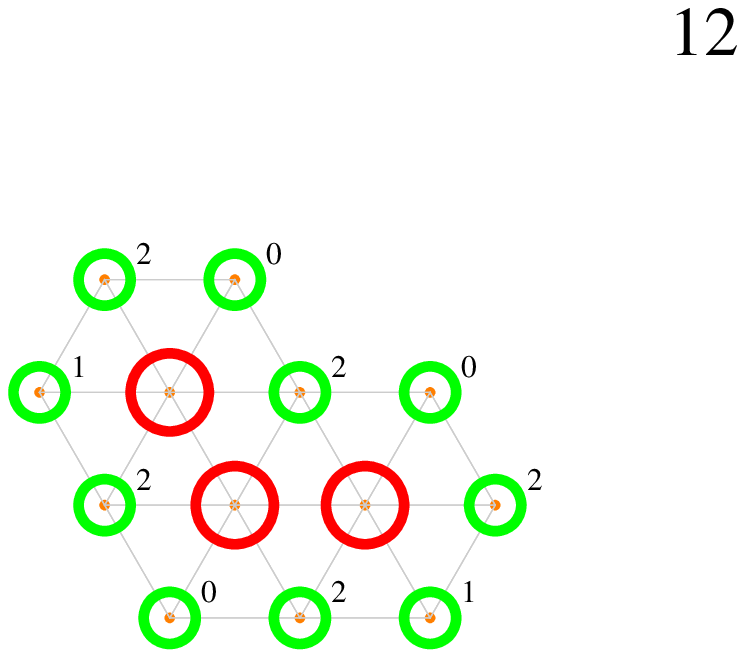}} 
\end{center}
{\bf Figure:} Reducing a one dimensional point at the boundary.
For any of the two situations, we end up with a region with one interior neighbor.

\begin{center}
\scalebox{0.35}{\includegraphics{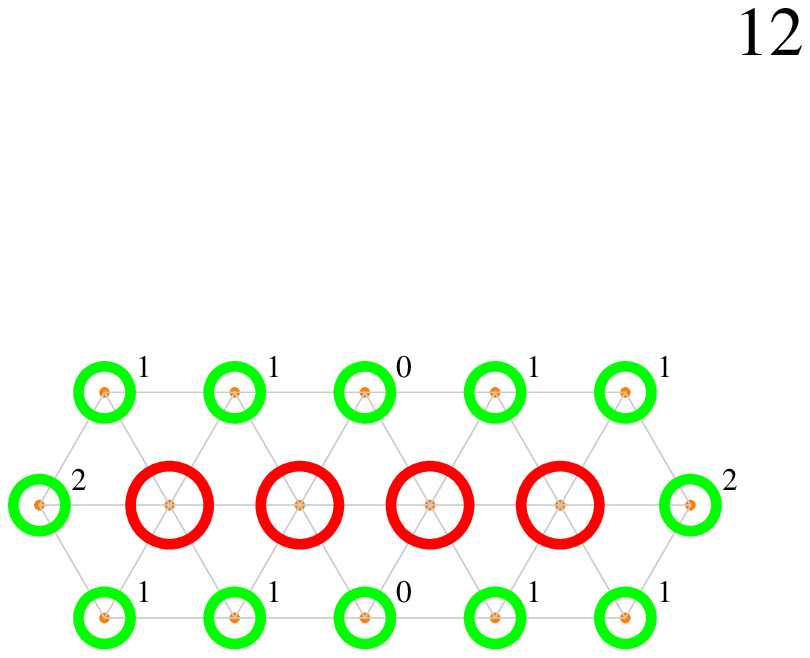}}
\scalebox{0.35}{\includegraphics{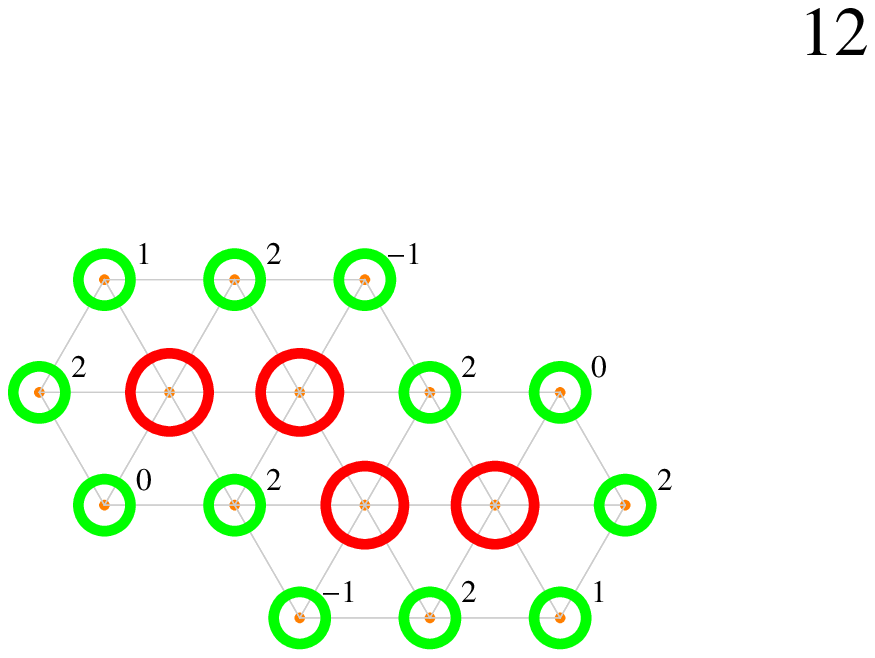}} 
\scalebox{0.35}{\includegraphics{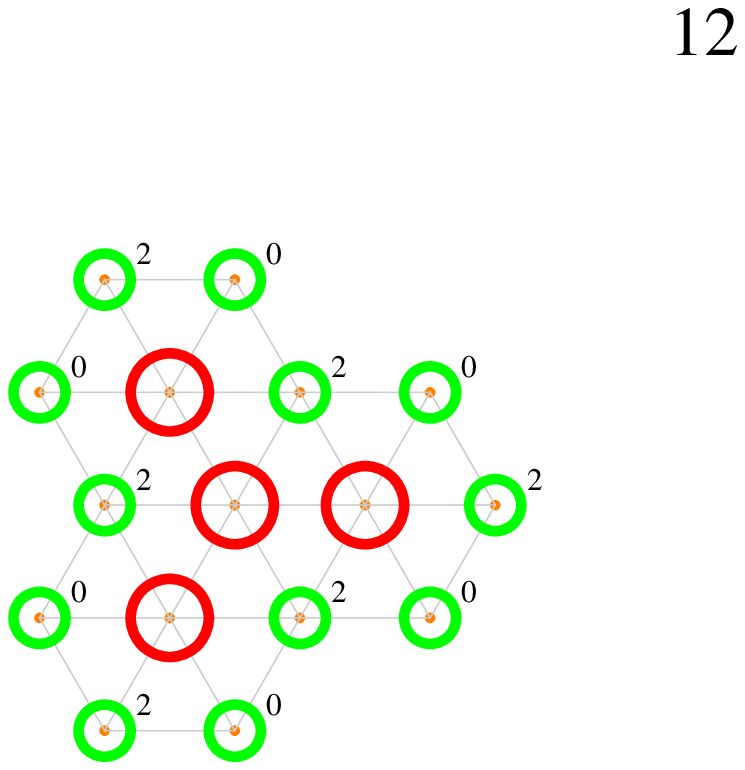}} 
\end{center}
{\bf Figure:} For any of the
first 3 situations, we end up with two neighboring interior points.  \\

After reducing one dimensional branches, the tree still can have one dimensional parts: these are
2D ridges connected with one-dimensional bridges which can not be pruned without changing the topology.  \\

\begin{center}
\scalebox{0.85}{\includegraphics{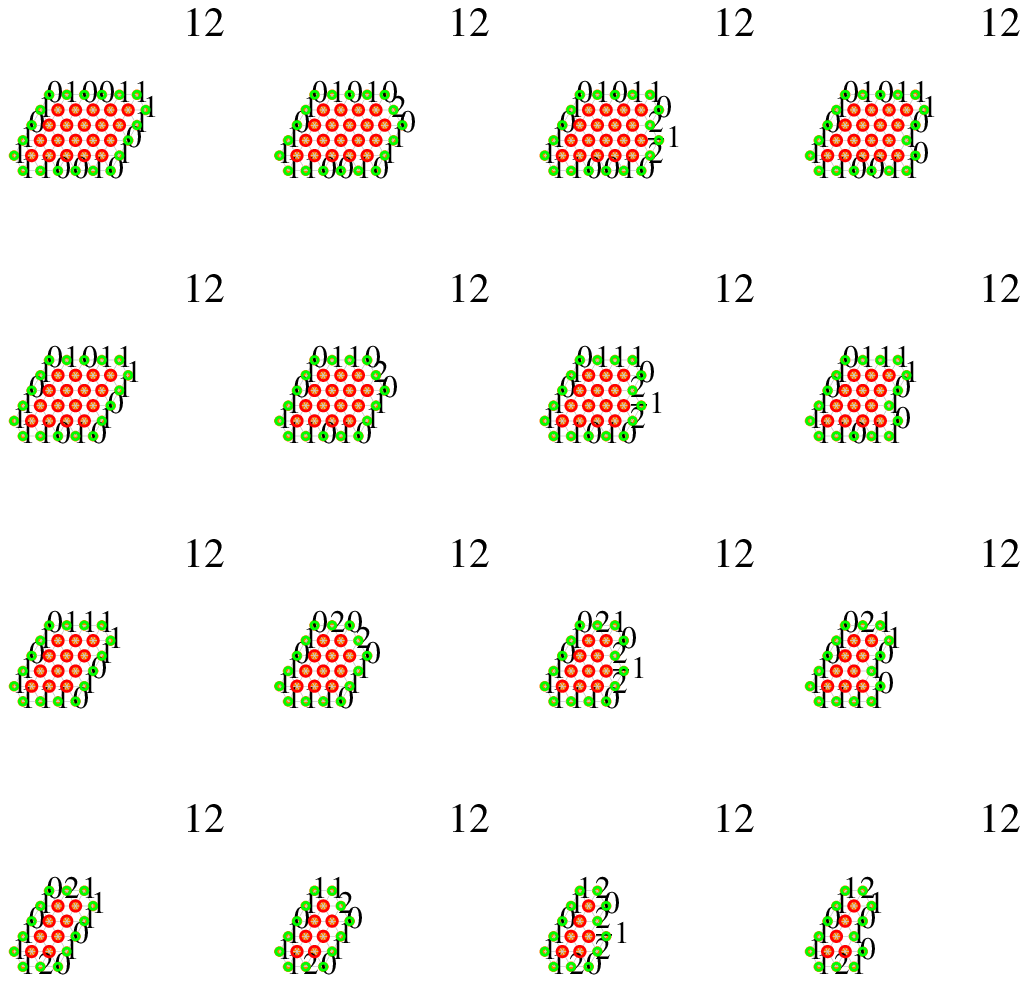}}
\end{center}
{\bf Figure:}  Etching thins out ridges. The etching is done at ridges which are end ridges, where only 
one bridge is attached.  With too many bridges attached, the etching process might not work. \\

The {\bf etching procedure} 
is invoked if no one-dimensional branches are left. The region consists now of 
two-dimensional ridges connected with bridges. 
Our goal is to see that we can remove a two-dimensional interior point of a ridge.  \\

The simply connectivity 
implies that there is a ridge which has only one bridge connected to it. To see this, look at a new
graph, which contains the two-dimensional ridges as vertices and one-dimensional bridges as edges.
This graph has no closed loops and is connected and must be a {\bf tree} with at least one end points. We can consequently 
focus our discussion to such an end-ridge for which only one $1$-dimensional bridge is attached. 
We are able to remove a boundary point on the opposite side of that region, where no branches can be and
where the boundary is "smooth". \\

\begin{center}
\scalebox{0.55}{\includegraphics{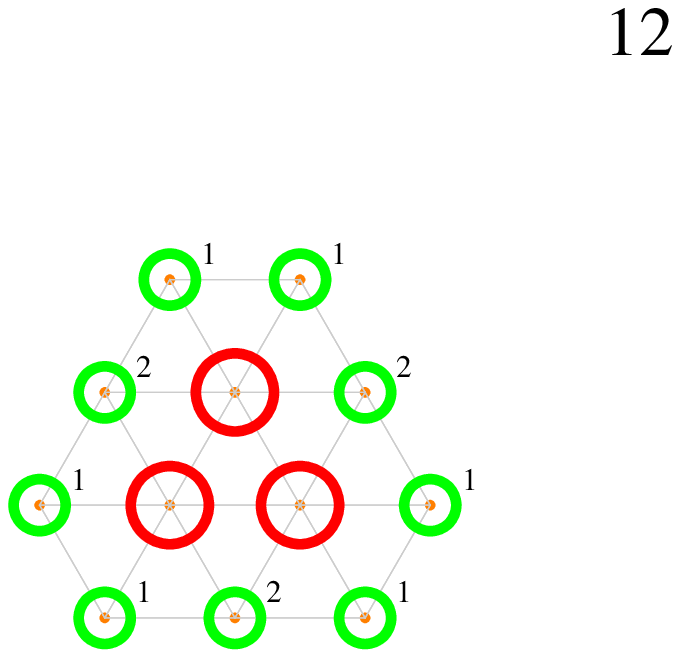}} 
\end{center}
{\bf Figure:} Reducing a two-dimensional interior point at the boundary 
which has 2 interior points as neighbors. \\

\begin{center}
\scalebox{0.55}{\includegraphics{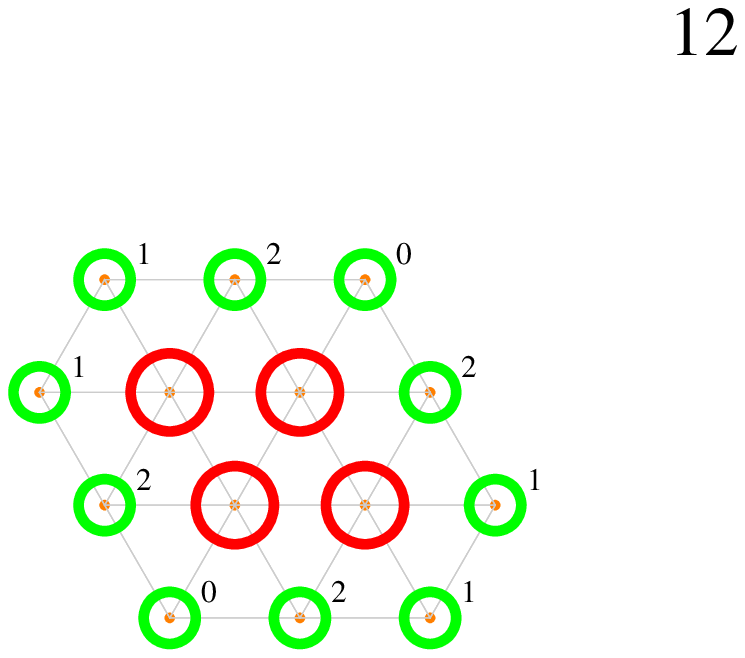}}
\end{center}
{\bf Figure:} Reducing a two-dimensional point at the boundary 
which has 3 interior points as neighbors. \\

\begin{center}
\scalebox{0.55}{\includegraphics{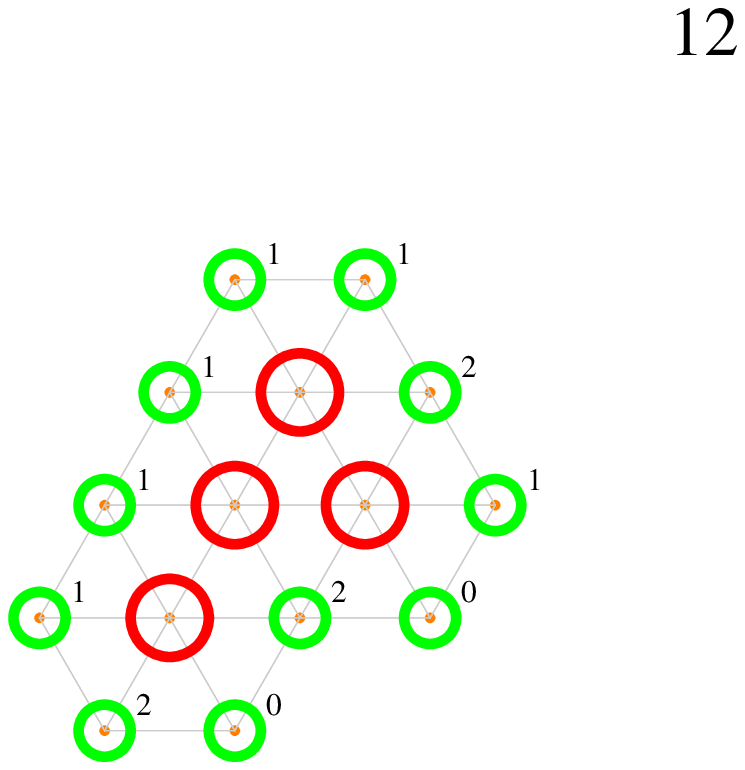}} 
\end{center}
{\bf Figure:} A situation where the point has 3 interior neighbors
and where the point can not be reduced. \\

\begin{center}
\scalebox{0.55}{\includegraphics{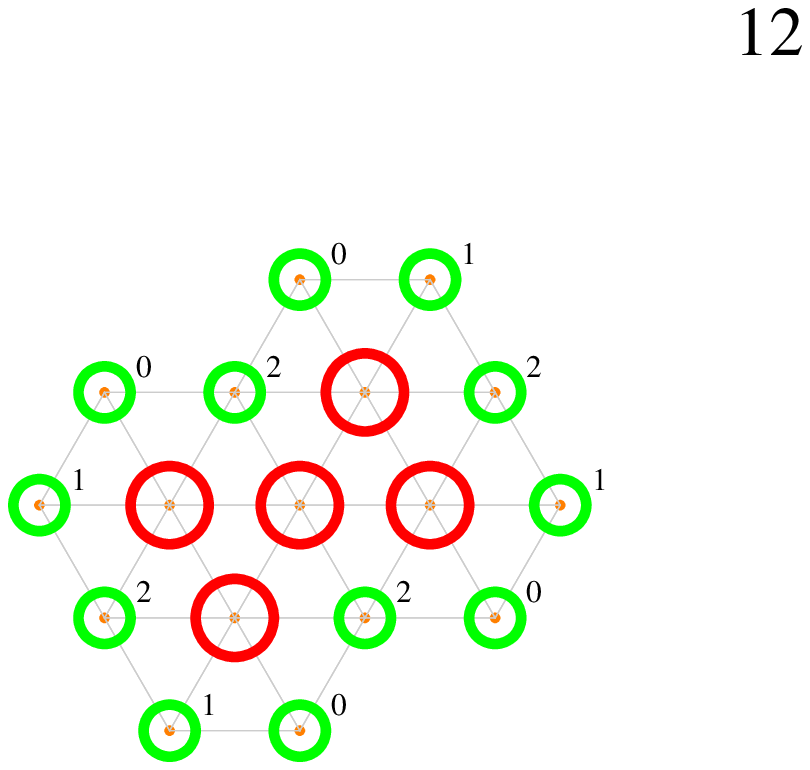}}
\end{center}
{\bf Figure:} A situation, where the point has 4 interior neighbors and where 
the point can not be removed. \\

\begin{center}
\scalebox{0.55}{\includegraphics{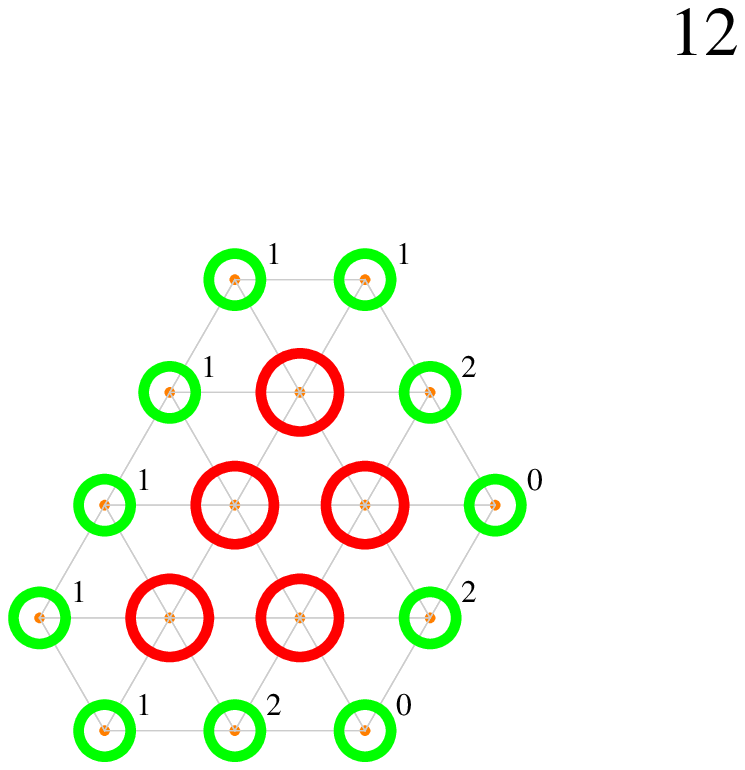}} 
\end{center} 
{\bf Figure:} A situation, where the point has 4 neighbors which are interior points and where 
the point can not be removed. \\

\begin{center}
\scalebox{0.55}{\includegraphics{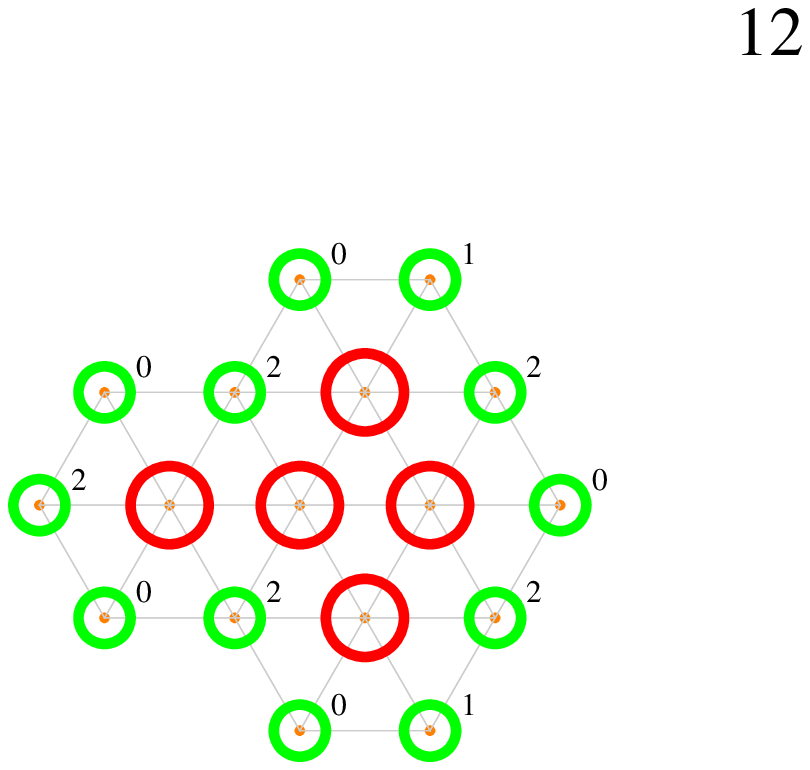}}
\end{center}
{\bf Figure:} Four interior points bounding an interior point. The middle point 
can not be removed while keeping the region a smooth region. \\

\begin{center}
\scalebox{0.55}{\includegraphics{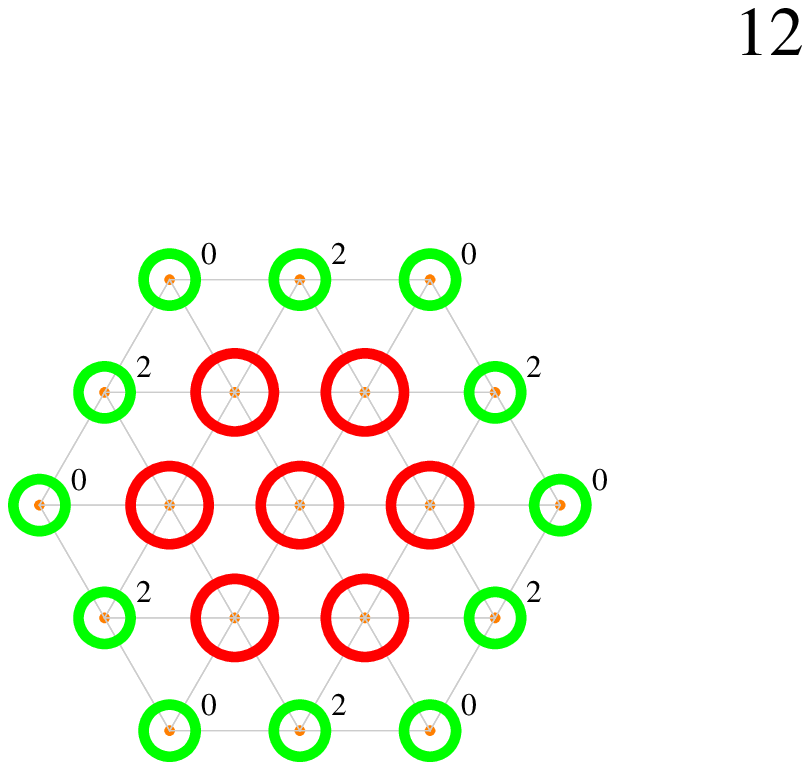}} 
\end{center}
{\bf Figure:} Having exactly 5 interior points bounding an interior point 
is not possible. We then necessarily have 6 neighbors.\\

\begin{center}
\scalebox{0.55}{\includegraphics{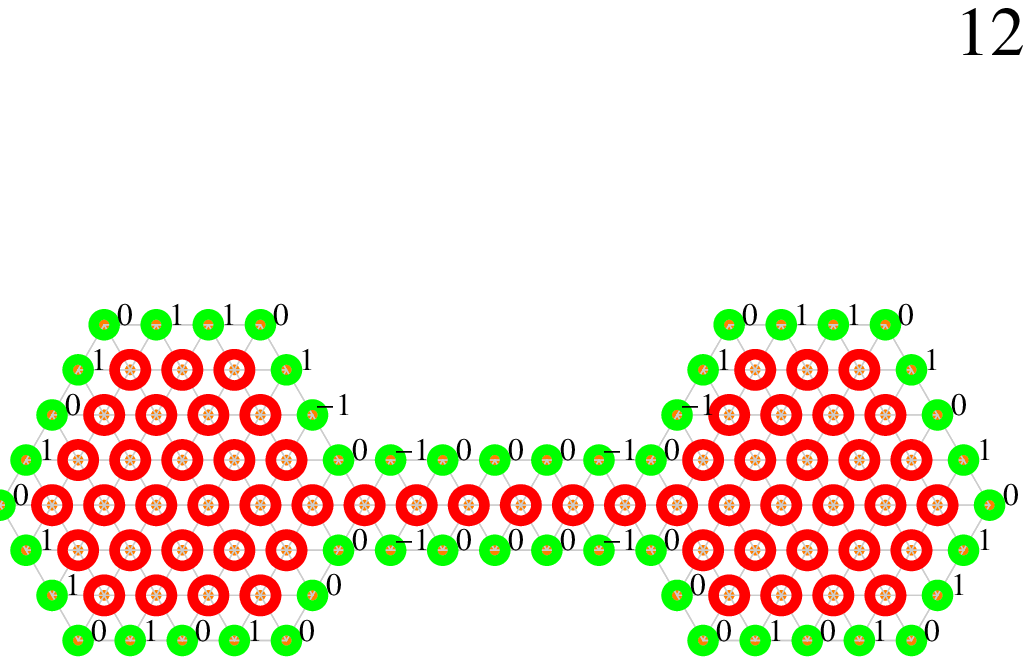}} 
\scalebox{0.55}{\includegraphics{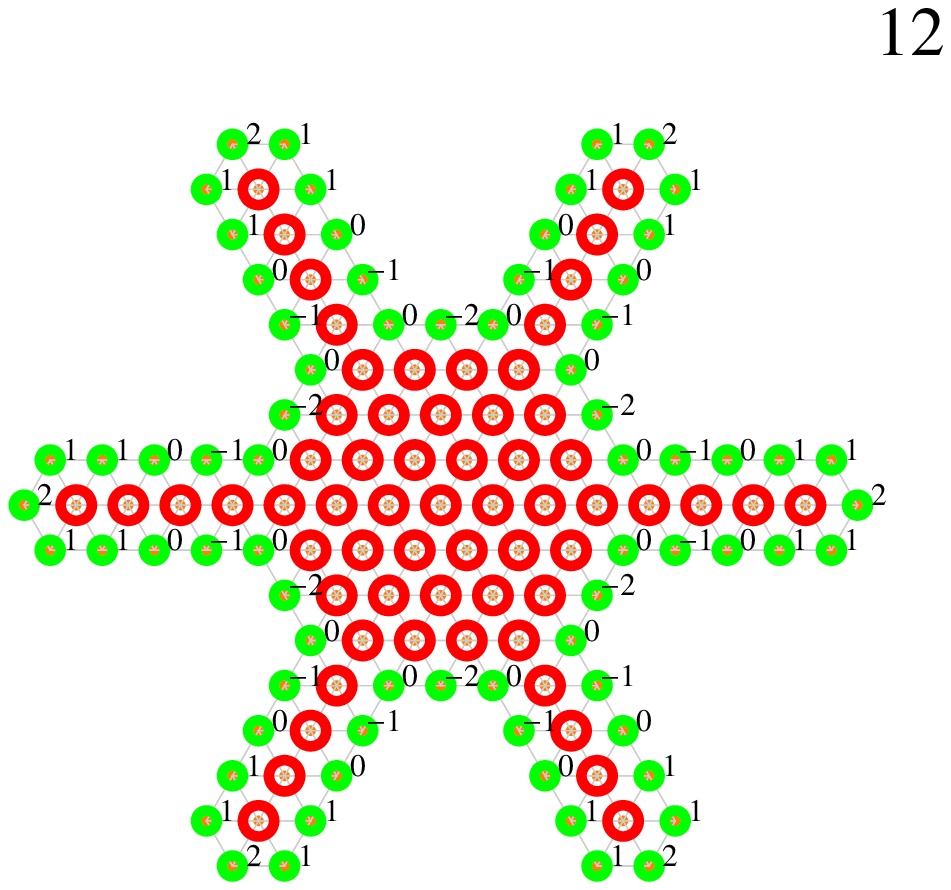}} 
\end{center}
{\bf Figure:} A bridge and branches. For the picture with the bridge, no interior point which is 
one-dimensional in ${\rm int}(G)$ can be removed. 
For the picture with the branches, no interior point which is 
two-dimensional in ${\rm int}(G)$ can be removed. This is a situation, where the branches 
first need to be trimmed.  \\

Once the etching process is over, we can again start pruning branches, or we are 
left with a region with only one interior point.
If a region $G$ can no more be pruned and edged then $H$ consists of only one point and $G$ 
consists of only $7$ points and in this case, we know the total curvature is $12$. \\

Since pruning and etching did not change the curvature and we have demonstrated that one can
reduce down every simply connected region to a situation with one interior point, 
this completes  the proof of the curvature 12 theorem. 
\end{proof} 

\section{Discrete Gauss-Bonnet}

To generalize the Umlaufsatz to domains which are not necessarily simply connected we
first {\bf define} the Euler characteristic of a region using Euler's formula: 

\begin{defn}
A {\bf face} in a domain $G$ is a triangle $(p,q,r)$ of 3 points in $G$ for which all three points 
have mutual distance $1$. An {\bf edge} in $G$ is a pair $p,q$ of points in $G$ of distance $1$. 
A {\bf vertex} is a point in $G$. Denote by $f$ the number of faces in $G$, by $e$ the number of 
edges and $v=|G|$ the number of vertices. The {\bf Euler characteristic} $\chi(G)$ of the 
domain $G$ is defined as $\chi(G) = v-e+f$. 
\end{defn}

Example: for a simply connected region, the Euler characteristic is $1$. 

\begin{lemma}
The Euler characteristic does not change under the pruning and etching operations defined above: 
both removing an end point of a one dimensional branch, as well as removing a two dimensional 
point from a ridge does not change it. 
\end{lemma}

\begin{proof}
The number of interior points of a smooth region is 
$2f-e+\chi$ and which can be proved by adding faces:
each face added is equivalent to adding 2 edges.
\end{proof}

\begin{center}
\scalebox{0.55}{\includegraphics{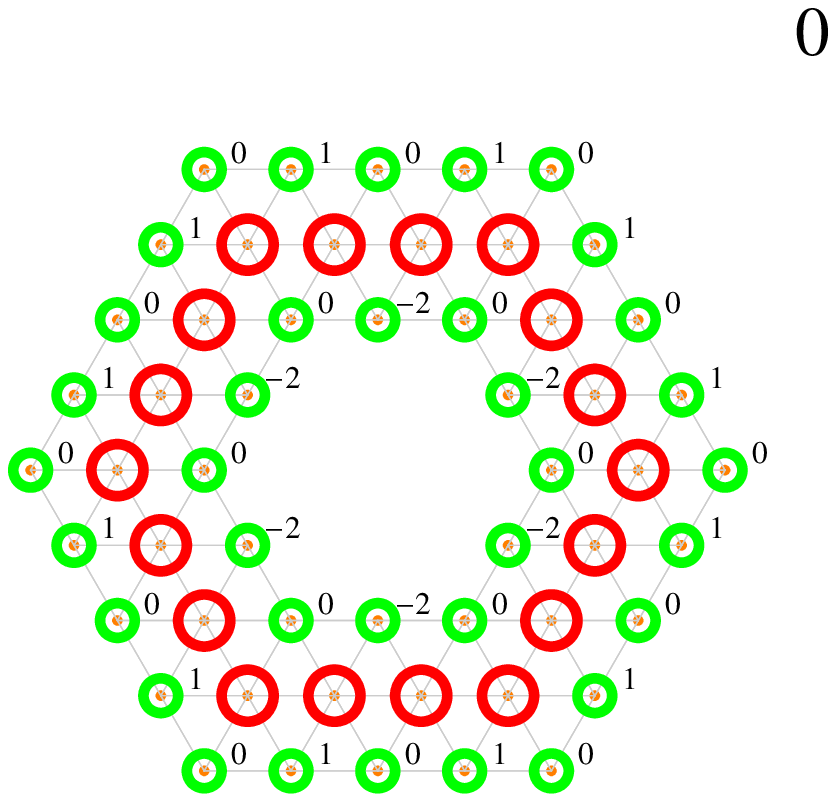}}  \\
\end{center}
{\bf Figure:}  A region, where no interior point can be removed any more and which has more than
one interior point is not simply connected. \\

{\bf Remark.} The Euler characteristic of ${\rm int}(G)$ and $G$ is the same if $G$ is a smooth region. 

\begin{thm}[Discrete Gauss-Bonnet theorem]
If $G$ is a finite smooth domain $G$ with boundary $C$, then
$$  \sum_{p \in C} K(p)  = 12 \chi(G)  \; . $$
\end{thm}

We could use the same pruning-etching technique as before. However, pruning and etching can lead
to final situations which have no end points like a ring. Instead of classifying
all these final situations, it is easier
to reduce the general situation to a simply connected situation. \\

There are two ways, how to change the topology: 
\begin{itemize}
\item build bridges between different connected components. 
\item fill holes to make the region simply connected
\end{itemize}
Merging different unconnected components is no problem. As long as their complement is a smooth 
region too, both the Euler characteristic as well as the total curvature add up.  \\

1. We can assume the region to be connected, because
both curvature as well as Euler characteristic are additive with respect to adding disjoint domains. 
To illustrate this more, we can also join two separated regions along with a one dimensional bridge.
The curvature drops by $12$, the number of connected components drops by $1$. 

\begin{center}
\scalebox{0.55}{\includegraphics{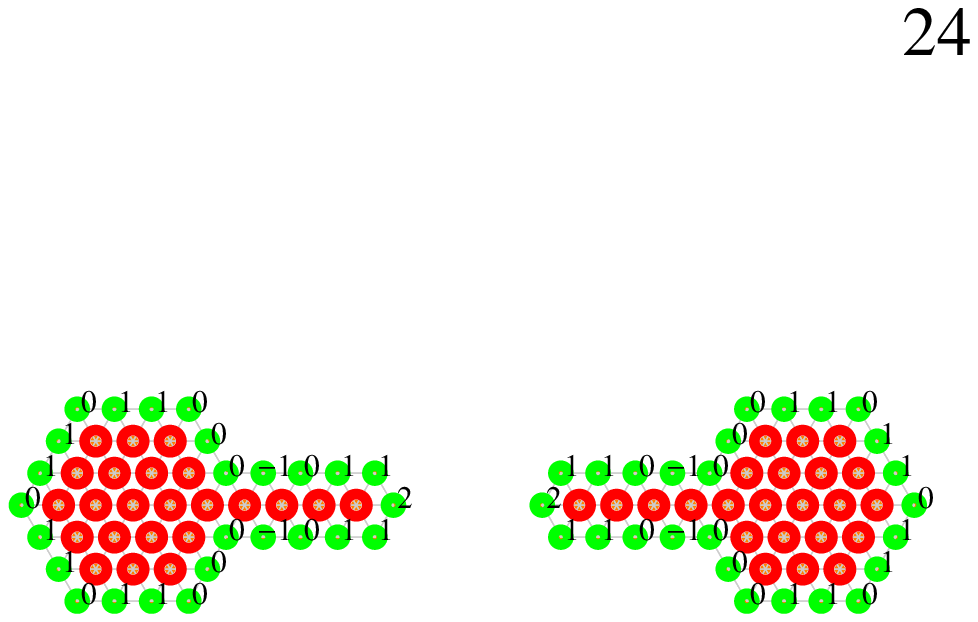}}
\scalebox{0.55}{\includegraphics{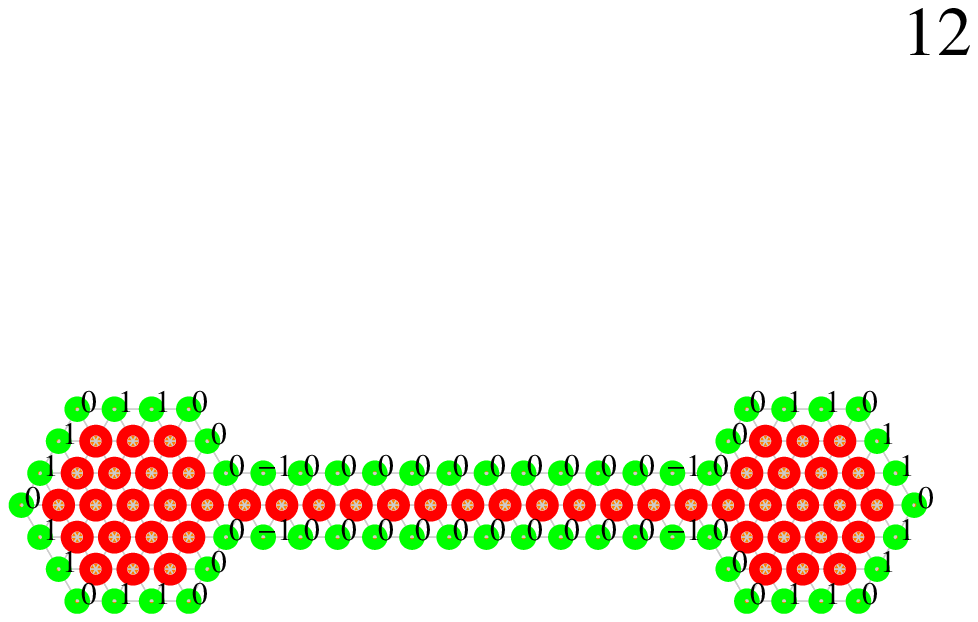}}
\end{center}
{\bf Figure:} Joining two regions changes the total curvature by 12. \\

\begin{defn}
The interior of a {\bf hole} $W$ is a bounded simply connected smooth region such that 
${\rm int}(W)$ is a component of the complement of $G$. By definition a whole $W$ and the region $G$
share a common part of the boundary. 
\end{defn} 

By the Umlaufsatz for simply connected regions, the hole has total curvature 12. A key observation is that
the point-wise curvatures at the inner boundary of $G$ enclosing the hole are just the negative of the 
corresponding point-wise curvatures of the hole. This follows almost from the definition of curvature 
and the fact that the circles $|S_1(p) \cap W| + |S_1(p) \cap G| = 6$ and 
$|S_2(p) \cap W| + |S_2(p) \cap G| = 12$ so that $K_{W}(p) + K_{G}(p) = 0$. 

\begin{center}
\scalebox{0.55}{\includegraphics{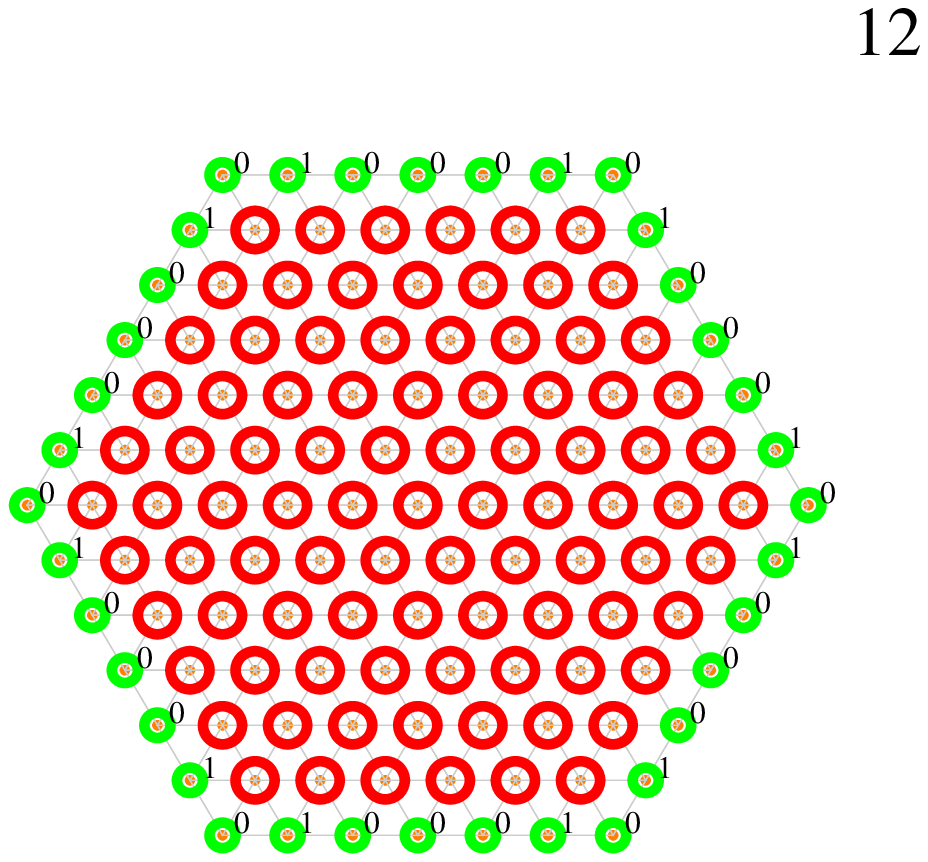}}
\scalebox{0.55}{\includegraphics{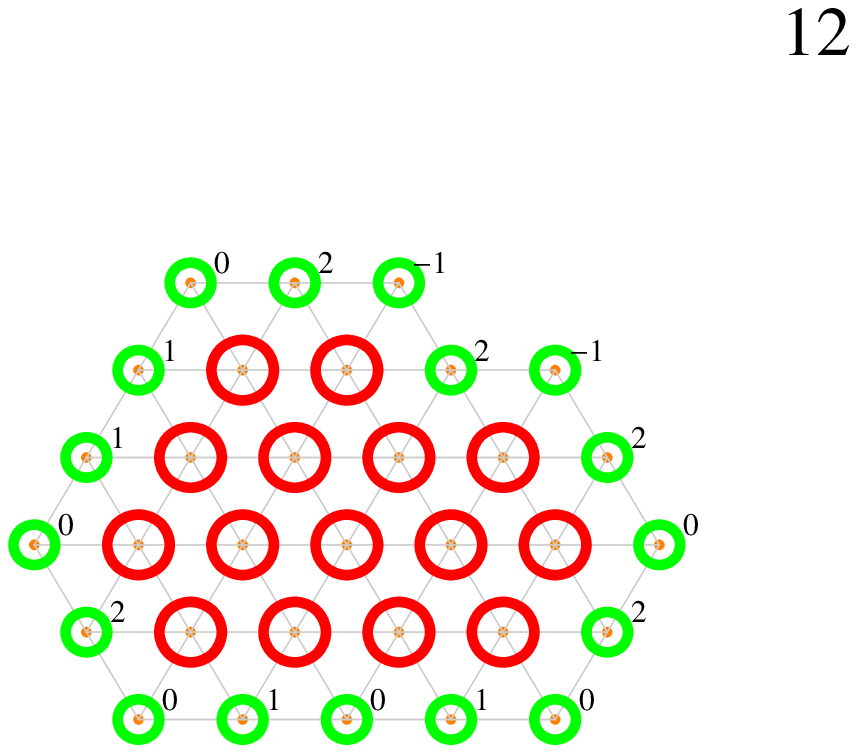}}

\scalebox{0.85}{\includegraphics{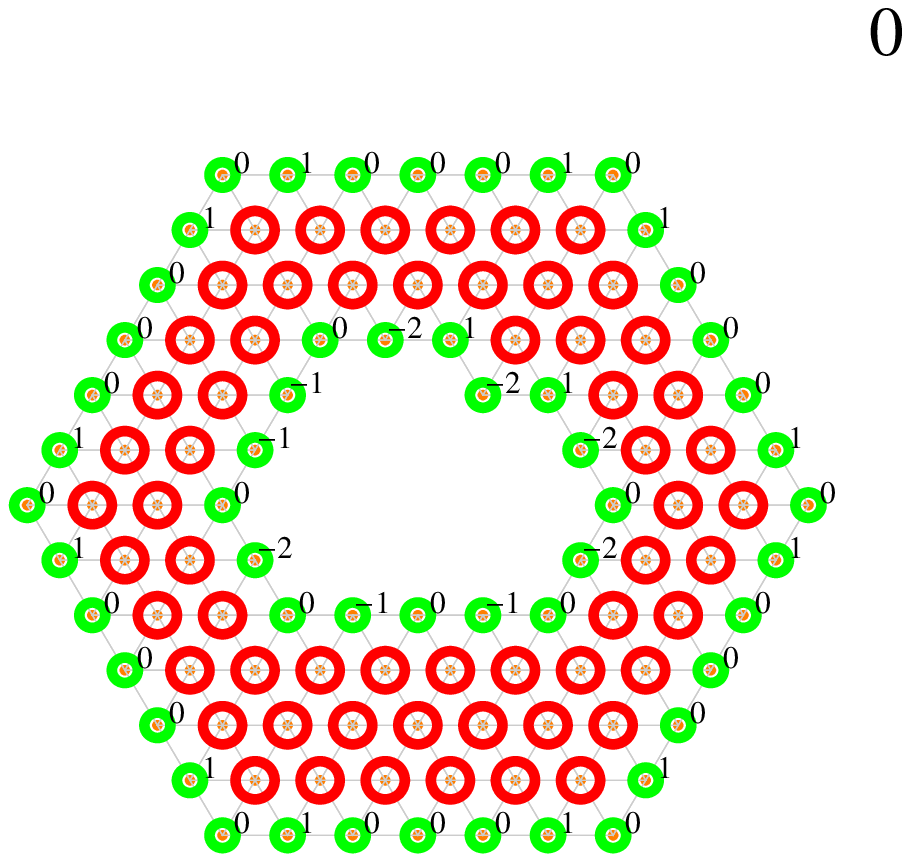}}
\end{center}
{\bf Figure:} Filling  a simply connected hole from a larger region adds to the curvature exactly the 
same amount than the total curvature of the hole. The reason is that the point-wise curvatures of the removed 
inside region matches exactly the curvatures of the inner outside region
if the inside region and the outside region have a common one-dimensional boundary. \\

This shows that if we fill a hole, the total curvature increases by $12$. Simultaneously, the 
Euler characteristic increases by $1$.  \\

Alternatively, we could also cut rings: 

\begin{center}
\scalebox{0.55}{\includegraphics{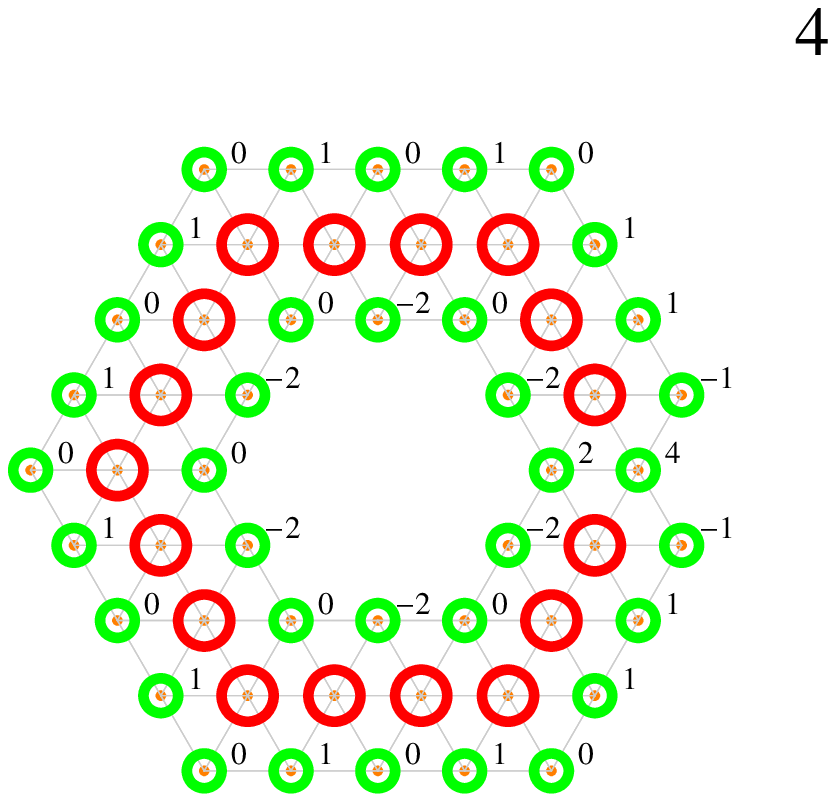}}
\scalebox{0.55}{\includegraphics{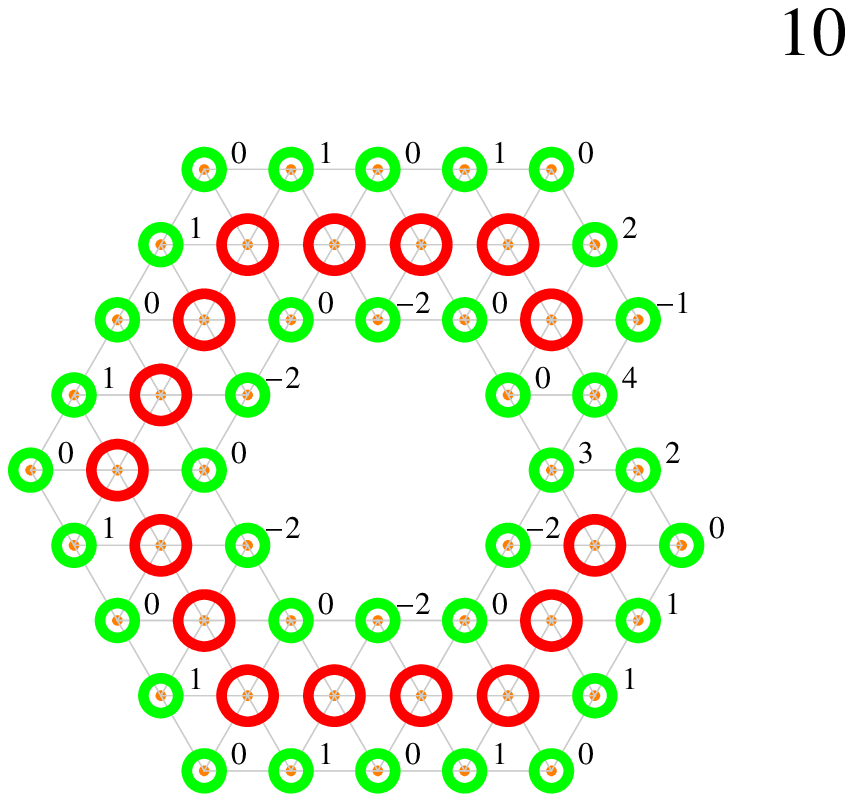}}
\end{center}

\begin{center}
\scalebox{0.55}{\includegraphics{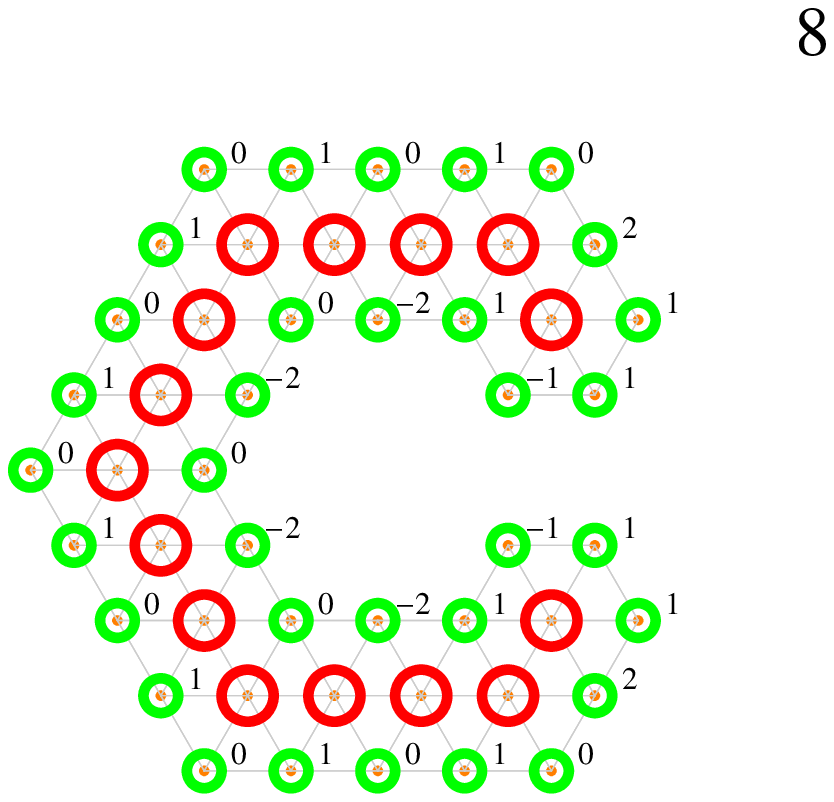}}
\scalebox{0.55}{\includegraphics{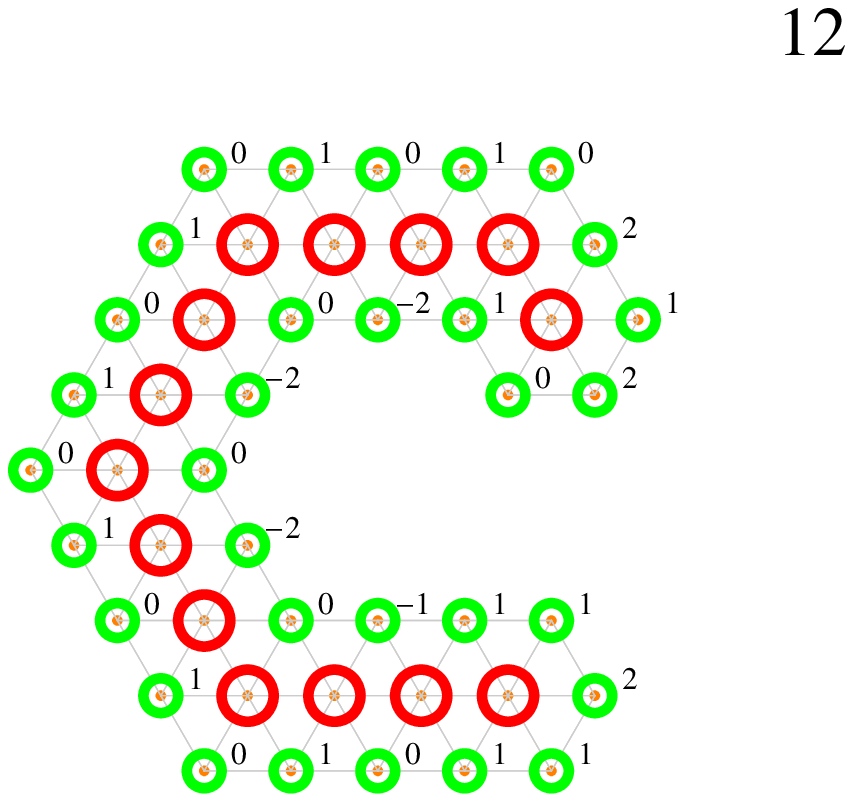}}
\end{center}
{\bf Figure:} When cutting a ring,  the curvature changes from 0 to 12. Only the
last region is a smooth region. The second last is a region but not smooth
because the complement is not a region. \\

\section{Compact flat graphs}

If we introduce identifications in the hexagonal background graph $X$, the topology of the background
space changes. Identifying points along two parallel lines for example produces a flat cylinder. 
With a triangular tiling, we can tessellate a torus. Because there is no boundary now, 
the sum of the curvatures is zero, which is the Euler characteristic. Note that there are many different
non-isometric graphs which lead to such tori. We call them {\bf twisted tori}. As {\bf graphs} they are 
different even if the number of faces, edges and vertices are fixed.

\begin{center}
\scalebox{0.40}{\includegraphics{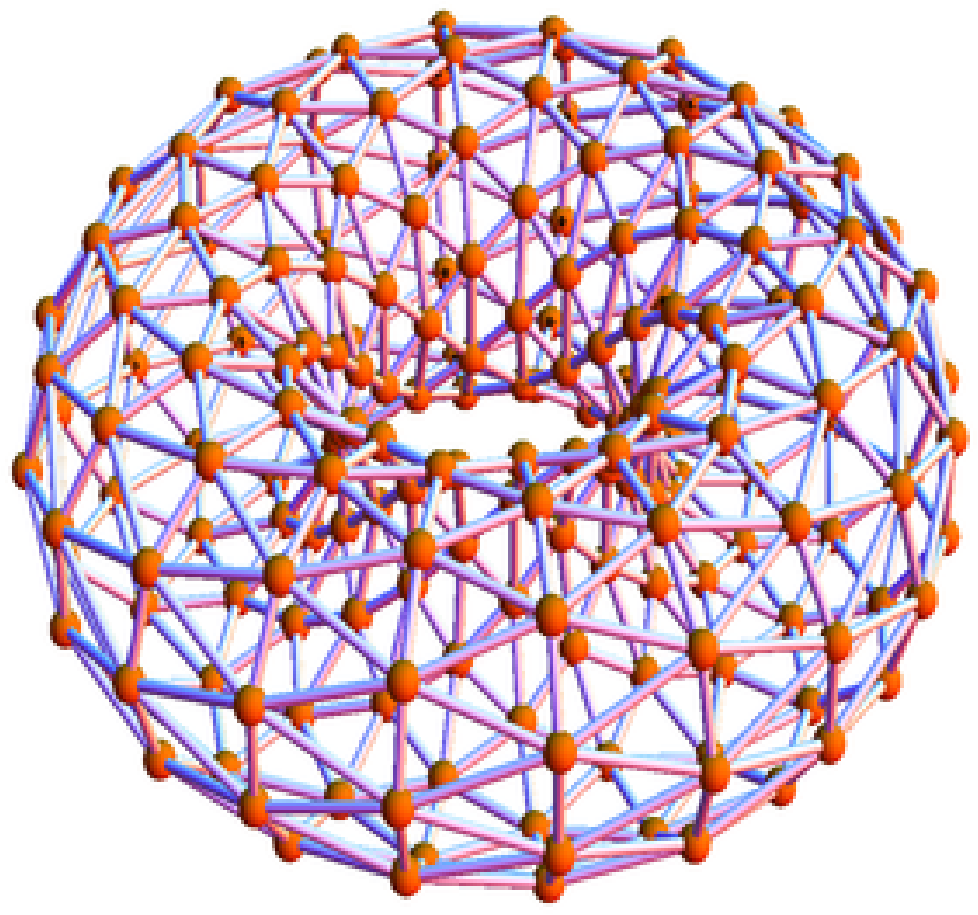}} \\
\end{center}
{\bf Figure:} A flat torus obtained by identifying opposite sides of a rectangular 
domain in a hex lattice.  The total curvature is zero.  \\

The notion of regular domain can be carried over to discrete manifolds like the twisted tori just
mentioned. Let $X$ be such a twisted background torus. We assume that it is large enough so that
$S_2(p)$ is a circle at every point. Let $G$ be a subgraph of $X$ defined as before. We still have:

\begin{thm}
If $G$ is a domain in a background torus $X$, then
$$ \sum_{p \in \delta G} K(p)  = 12 \chi(G)  \; . $$
\end{thm}

{\rm Remark}: 
More flat compact graphs can be obtained using "worm hole" constructions. Let $X$ be a possibly
twisted torus as defined above and let $p$,$q$ be two points for which the discs $B_{r}(p)$ and 
$B_r(q)$ are disjoint and the spheres $S_r(p)$ and $S_r(q)$ are circles. Any orientable graph isomorphism 
between $S_r(p)$ and $S_r(q)$ produces an identification  of points in $G = X \setminus B_{r-1}(p) \cup S_r(p)
\setminus B_{r-1}(q) \cup S_r(q)$. Without identification,  the total curvature of the boundary of the 
domain $G$ is $\sum_{p \in C} K(p)  = 12 \chi(G) = -24$ because every removed disc produces curvature $-12$. 

\section{Combinatorial curvature}

In this section, we consider a more elementary Gauss-Bonnet formula. 
The curvature is again defined by a Puisaux discretization, but only circles of length $1$ appear
in the definition. It is a first order curvature. 

\begin{defn}
For a two dimensional graph with boundary, we define the {\bf combinatorial Puiseux curvature} as
$$  K_1(g) = 6-|S_1(g)| $$
for interior points and 
$$  K_1(g) = 3-|S_1(g)| $$ 
for boundary points. 
\end{defn}

For this combinatorial curvature, Gauss-Bonnet is much  easier. For subgraphs 
of hexagonal lattice
the boundary curvature is almost trivially equal to $6$ because the curvature is related to angles
of the corresponding polygon: for a boundary point, $K_1(p) \pi/3$ is the interior angle of the polygon. 
Because $\sum_p K_1(p) \pi/3 = 2\pi$ by the polygonal version of the Umlaufsatz, we have $\sum_p K_1(p) = 6$. \\

Actually, Gauss-Bonnet holds in great generally for arbitrary two-dimensional graphs with or without boundary. 
Since it is so closely related to the Euler characteristic, we should the attribute it to Euler,
even so we are not aware that Euler considered $K_1(g)$, nor that he looked at the dimension of a graph. \\


\begin{thm}[Combinatorial Gauss-Bonnet]
Assume $G$ is a two-dimensional finite graph for which the boundary is either empty or forms itself a one dimensional set. 
Then 
$$  \sum_{g \in G} K_1(g) = 6 \chi(G)  \; . $$ 
\end{thm} 

{\bf Remarks.} \\
1) This result does not need the rigid requirements on the "domain" as before nor does the graph have to be part of $X$; 
it works for any two-dimensional graph, with or without boundary. \\
2) The result appears in a different formulation, which does not make its Gauss-Bonnet nature evident:
the Princeton Companion to Mathematics \cite{princetonguide} mentions on page 832 the formula
$\sum_n (6-n) f_n = 12$,  where $f_n$ is the number $n$-hedral faces and summation is
over all faces. This is an equivalent formulation, but it makes the Gauss-Bonnet character 
less evident. \\

Note that the combinatorial Gauss-Bonnet theorem is {\bf entirely graph theoretical}. 
It avoids the pitfalls with the definition, what a polyhedron is \cite{lakatos}.
(Common definitions of "polyhedra" refer to an ambient Euclidean space or 
impose additional structure on a graph).
We can take a general finite graph which is two-dimensional at each point. Its points
are either boundary points, points where the unit sphere is one-dimensional but not closed,
or an interior points, where the unit sphere is a circle, a simple closed graph without boundary.

\begin{proof}
Assume first that the graph $G$ has no boundary. For a two-dimensional graph, all faces necessarily 
are triangles. Therefore, the number of faces $f$ and the number of edges $e$ are related by the 
{\bf dimensionality formula}
$$  3f=2e  \; . $$ 
Furthermore, we have the {\bf edge formula}
$$  \sum_g |S_1(g)| = 2e $$ 
which is obtained by counting edges in a different way. 

\begin{center}
\parbox{13.8cm}{
\parbox{6.5cm}{\begin{center}\scalebox{0.28}{\includegraphics{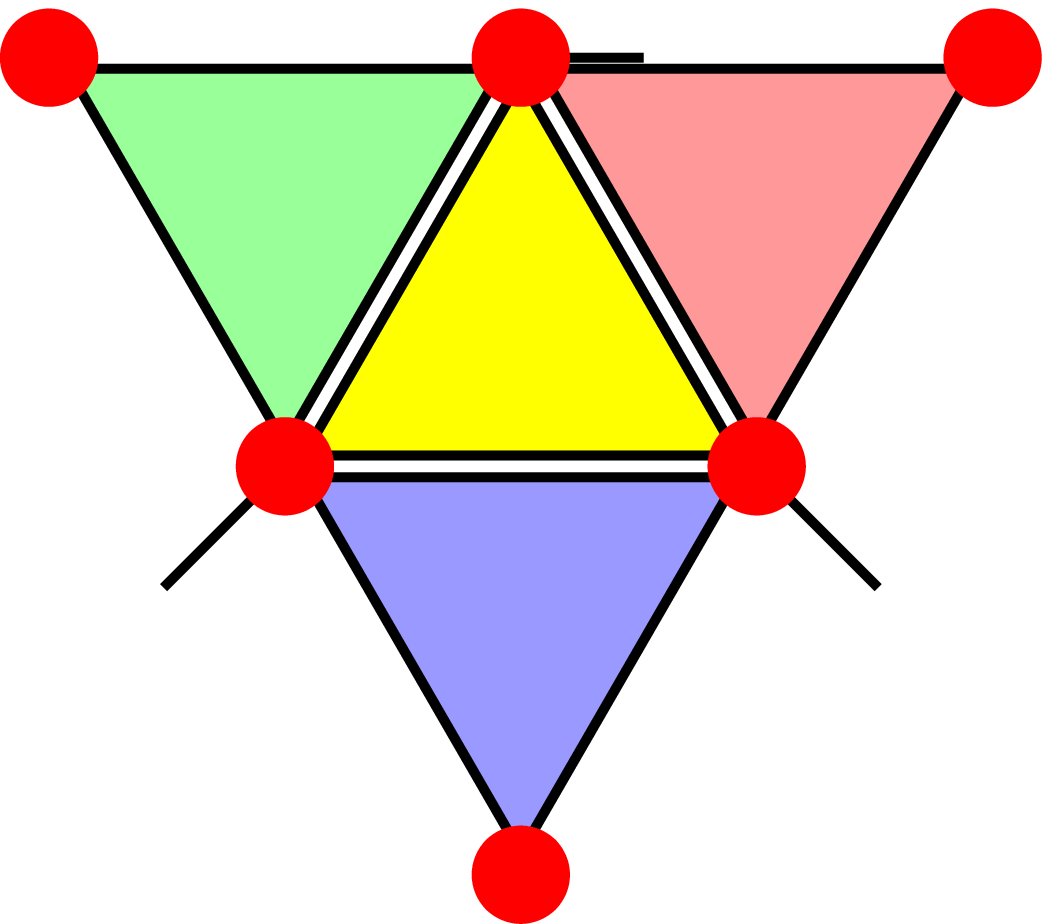}}\end{center}}
\parbox{6.5cm}{\begin{center}\scalebox{0.28}{\includegraphics{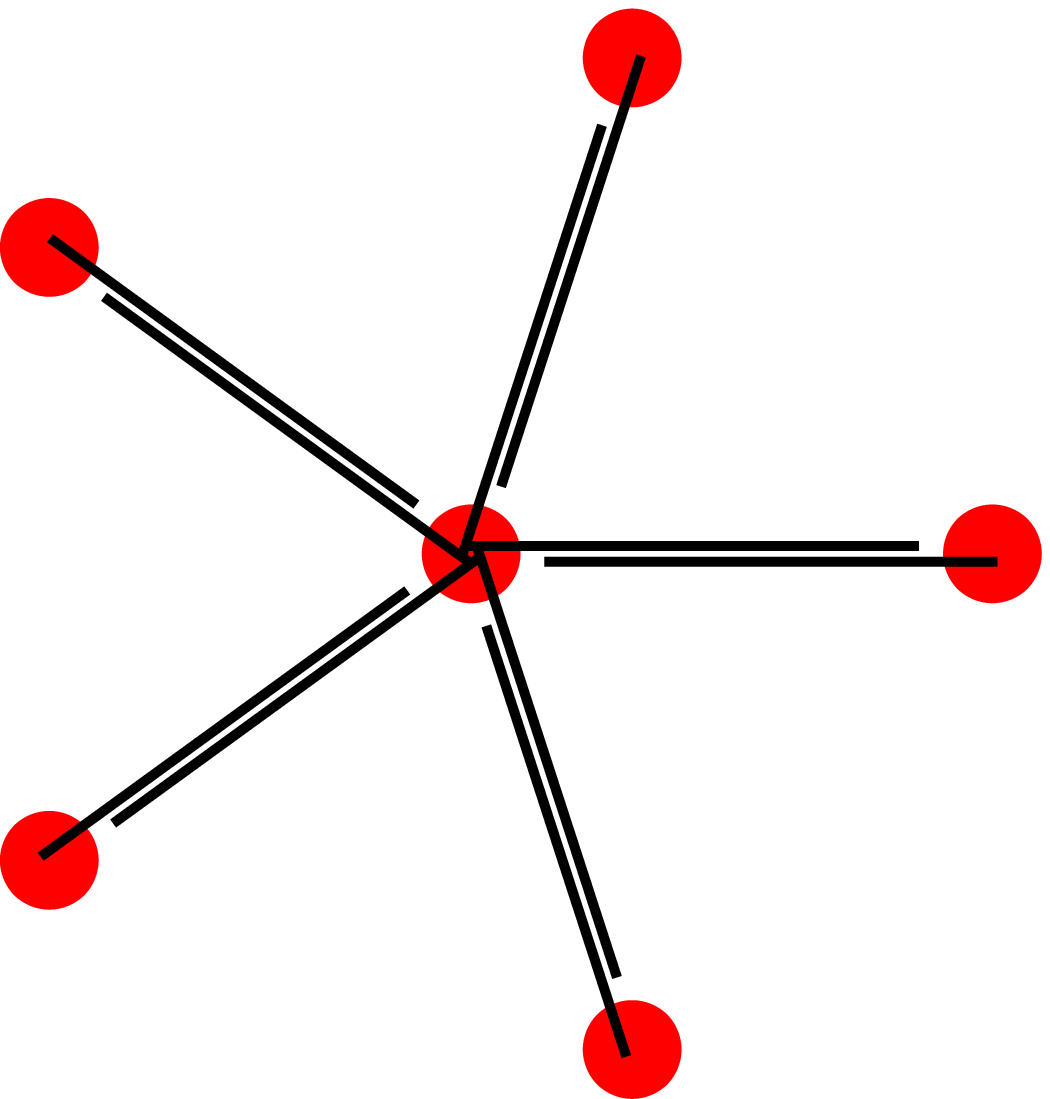}}\end{center}}
}
\end{center}
\parbox{13.8cm}{
\parbox{6.5cm}{{\bf Figure:} the identity $3f = 2e$.}
\parbox{6.5cm}{{\bf Figure:} the identity $\sum_g |S_1(g)| = 2d$.}
}

Using the definition of the Euler characteristic, and these two formulas,  we compute
\begin{eqnarray*}
 6 \chi(G) &=& 6 f - 6 e + 6 v = -2e + 6 v  \\
           &=& - \sum_{g \in G} |S_1(g)| + 6v = - \sum_{g \in G} (|S_1(g)|-6)=\sum_{g \in G} K_1(g) \; . 
\end{eqnarray*}
This finishes the proof in the case of a graph without boundary.  \\

The case with boundary can be reduced to the boundary-less case: 
The boundary is a union of closed cycles. For each of these $m$ cycles $\delta G_i$
just add an other point $P_i$ and add $n=|\delta G_i|$ edge connections from each 
of the cycle boundary points of $G_i$ to $P_i$. This produces a graph $H$ without boundary and 
which contains $G$ as a subgraph. The formula without boundary shows that
$6 \chi(H)$ is the sum of curvatures of the original interior points and 
the sum of the curvatures of the boundary points as well as the sum of the curvatures 
$K_1(P_i) = 6-n_i$ to each newly added point $P_i$:
$$ 6 \chi(H) = \sum_{g \in {\rm int}(G)} K(g) + \sum_{g \in \delta G} K_H(g) + \sum_i (6-n_i)  \; . $$
We also have
$$ \chi(H) = \chi(G) + m-|\delta G|+|\delta G| = \chi(G) + m  \; . $$
For boundary points, $K_H(g) = 6-|S_1(g)|-2$ and 
$K_G(g) = 3-|S_1(g)|$ so that $K_H(g)-K_G(g) = 1$. From the previous boundary less case, we get
$$ 6 \chi(H) = \sum_{g \in H} K_H(g) = \sum_{g \in {\rm int}(G)} K_G(g) + \sum_{i=1}^m \sum_{g \in \delta G_i} (K_G(g) +1) + (6-n_i)  
  =  \sum_{g \in G} K(g) + 6m \;  $$
so that 
$$ \chi(G) = 6 \chi(H)-6m = \sum_{g \in G} K(g) + 6m - 6m =  \sum_{g \in G} K(g) \; . $$ 
\end{proof}

The just verified combinatorial Gauss-Bonnet is entirely {\bf graph theoretical}.
Our curvature definition $K$ was motivated from the notion of {\bf Jacobi fields} 
in the classical case given by second derivatives. 
While "smoothness" requirements" are necessary for the more sophisticated
Gauss-Bonnet formula, the just mentioned metric Gauss-Bonnet holds for any
polyhedron with triangular faces. For example, every finite triangularization of a two-dimensional compact manifold works. 
Let us explain a bit more, why the curvature
$$  K = 2 |S_1| - |S_2| $$ is "differential geometric":
the {\bf Gauss-Jacobi equations} $f'' = -K f$ in differential
geometry require the second differences of a Jacobi field $f$. 
\cite{BergerPanorama}.
Our starting point had been to extend Jacobi fields  to the discrete for numerical purposes:
for a discretized Jacobi field with smallest space step $1$, we have $f''(0) = f(2) - 2 f(1) + f(0)
 = f(2) - 2 f(1)$. The Jacobi equations suggest to call this $-K$. Since $f(k)$ is the variation of the
geodesic when changing the angle, we can integrate over the circle and we get the length $|S_k|$ of the circle
of radius $k$. Therefore $K$ is a multiple of $2 |S_1|-|S_2|$ and there is no reason to normalize this in the discrete. \\
The "first order curvature" $K_1=6-|S_1|$ on the other hand only requires first order differences. 
The curvature $K$ has some advantages over the curvature $K_1$: \\

\begin{itemize}
\item the curvature formula for the boundary and in the interior is the same, while for the curvature $K_1$,
   one has to distinguish boundary and interior.  \\
\item there is no reference to a flat background structure for $K$, while $K_1$ refers to the flat situation with 
   via integers $6$ or $3$. \\
\item it can be generalized to more general situations, where the distance in the graph can vary and
   where we have no natural flatness as a reference. We can look for example for distance functions which 
   minimize the total curvature.\\
\item it is more closely rooted to differential geometry of manifolds and classical notions like
   Jacobi fields, a notion which is of "second order" too.  \\
\item it can be adapted to higher dimension, when defining scalar curvature for graphs and where
   no natural "flat triangulated ambient reference graph" exists.  \\
\end{itemize}

To summarize, we think that while $K_1$ is "metric", $K$ has a more "differential geometric" flavor. Similarly as
many metric results extend to the differential geometric setup, things are more restricted also in the 
discrete, if higher order difference notions are used. The limitations of the results are related to similar
limitations we know in the continuum. \\

We can combine the two results: for the Puiseux curvature with radius $2$ defined by
$$   K_2(g) = 12-|S_2(g)| \; , $$ 
we get the following corollary: 

\begin{coro}[K2 formula]
If $G$ is a two-dimensional smooth domain in the triangular tessellation $X$ of the plane, then 
$$ \sum_{g \in C} K_2(g) = 24 \chi(G) \; . $$
\end{coro}

\begin{proof}
Since $\sum_g 12-2 |S_1(g)| = 12 \chi(G)$ and $\sum_g 2 |S_1(G)| - |S_2(g)| = 12 \chi(G)$, we get by 
addition $\sum_g 12-|S_2(g)| = 24 \chi(G)$. The left hand side is the combinatorial Puiseux curvature
for radius $r=2$.
\end{proof}

Note that unlike the combinatorial curvature formula $\sum_{g \in G} K_1(g) = 6 \chi(G)$, the $K_2$ formula
is only obvious modulo the main result for "smooth domains" proved here. 
If we wanted to establish Gauss-Bonnet type results for curvatures like $K_3 = 3 S_1- S_3$, 
the restrictions on discrete domains would be even more severe. 

\vfill

\bibliographystyle{plain}

\vspace{1cm}

\end{document}